\documentclass[11pt,leqno]{article}

\usepackage{amsthm,amsfonts,amssymb,amsmath,oldgerm}
\usepackage{fullpage}
\usepackage{graphicx}
\usepackage{mathrsfs}
\usepackage{color}

\numberwithin{equation}{section}

\renewcommand\d{\partial}

\renewcommand\o{\omega}

\renewcommand\th{\theta}
\def\l{\lambda}
\def\eps{\varepsilon }
\newcommand\R{\mathbb R}

\renewcommand\L{\mathcal{L}[\phi]}

\newcommand\G{G[\phi]}
\newcommand\Ginv{G^{-1}[\phi]}
\newcommand\Gadj{G^\dagger[\phi]}

\newcommand\A{A[\phi]}

\newcommand\LA{\left\langle}
\newcommand\RA{\right\rangle}

\newcommand{\RM}{{\mathbb{R}}}
\newcommand{\CM}{{\mathbb{C}}}

\newtheorem{theorem}{Theorem}[section]

\newtheorem{corollary}[theorem]{Corollary}
\newtheorem{lemma}[theorem]{Lemma}
\newtheorem{remark}[theorem]{Remark}
\theoremstyle{definition}

\allowdisplaybreaks[3]

\title{Modulational Instability of Viscous Fluid Conduit Periodic Waves}

\author{Mathew~A.~Johnson\thanks{Department of Mathematics, University of Kansas, 1460 Jayhawk Boulevard, 
Lawrence, KS 66045; matjohn@ku.edu}\quad\&\quad Wesley R. Perkins\thanks{Department of Mathematics, University of Kansas, 1460 Jayhawk Boulevard, 
Lawrence, KS 66045; wesley.perkins@ku.edu} }

\date{\today}

\begin{document}

 \maketitle

\begin{abstract}
In this paper, we are interested in studying the modulational dynamics of interfacial waves rising buoyantly along a conduit of a viscous liquid.
Formally, the behavior of modulated periodic waves on large space and time scales may be described through the use of Whitham modulation theory.
The application of Whitham theory, however, is 
based on formal asymptotic (WKB) methods, thus removing a layer of rigor that would otherwise support their predictions. In this study, 
we aim at rigorously verifying the predictions of the Whitham theory, as it pertains to the modulational stability of periodic waves, in the context of the so-called conduit equation, a 
nonlinear dispersive PDE governing the evolution of the circular interface separating a light, viscous fluid rising buoyantly through a heavy, 
more viscous, miscible fluid at small Reynolds numbers.  In particular, using rigorous spectral perturbation theory, we connect the predictions 
of Whitham theory to the rigorous spectral (in particular, modulational) stability of the underlying wave trains.  
This makes rigorous recent formal results on the conduit equation obtained by Maiden and Hoefer.
\end{abstract}

\section{Introduction}  \label{S:intro}

In this paper, we consider the modulation of periodic traveling wave solutions to the conduit equation
\begin{equation}\label{e:conduit1}
u_t + (u^2)_x - (u^2(u^{-1}u_t)_x)_x = 0,
\end{equation}
which was derived in \cite{OC86} to model the evolution of a circular interface separating a light, viscous fluid rising buoyantly through a heavy, more viscous, miscible
fluid at small Reynolds numbers.  In \eqref{e:conduit1}, $u=u(x,t)$ denotes a nondimensional cross-sectional area of the interface at nondimensional vertical coordinate
$x$ and nondimensional time $t$: see Figure \ref{f:Schematic}(a).  The conduit equation \eqref{e:conduit1} has also been studied in the geological context, where it is known to describe,
under appropriate assumptions, the vertical transport of molten rock up a viscously deformable pipe (for example, narrow conduits and dykes) in the earth's crust.  
In that context, \eqref{e:conduit1} is a special case of the more general ``magma" equations \cite{SS84,SSW86}
\begin{equation}\label{e:magma}
u_t + (u^n)_x - (u^n(u^{-m}u_t)_x)_x = 0:
\end{equation}
where here  the parameters $n$ and $m$ correspond to permeability of the rock and the bulk viscosity, respectively.  The physical regime for these exponents 
is $2\leq n\leq 5$ and $0\leq m\leq 1$: see \cite{SS84}.  Clearly, the conduit equation corresponds to \eqref{e:magma} with
$(n,m)=(2,1)$.    

In contrast to magma, however, viscous fluid conduits are easily accessible in a laboratory setting: see, for example, \cite{MLASH16} and references therein.  Consequently,
there has been quite a bit of study recently into the dynamics of solutions of the conduit equation \eqref{e:conduit1} and their comparison to laboratory experiments.
As described in \cite{MH16}, early experimental studies of viscous fluid conduits concentrated primarily on the formation of the conduit itself 
via the continuous injection of an intrusive viscous fluid into an exterior, miscible, much more viscous fluid\cite{OC86}.  
Since then, a considerable amount of effort has been spent studying the dynamics and stability of solitary waves, as well as soliton-soliton interactions \cite{OC86,HW90,SW08,LHE14}.
More recently, it has been  observed experimentally that the competition of dispersive effects due to buoyancy  
and the nonlinear self-steepening effects of the surrounding media may result in the formation
of dispersive shock waves (DSWs): see, for example, \cite{MLASH16}.  As described there, by adjusting the injection rate of the intrusive viscous fluid appropriately 
it was found that interfacial wave oscillations form behind a sharp, soliton-like leading edge, with the wider regions moving faster than narrower regions: see Figure \ref{f:Schematic}(b).  
Such patterns correspond to dispersively regularized shock waves and have been the subject of much recent study due to their experimental realization \cite{SSW86,AH09,WH90}.
Consequently, spatially modulated oscillations seem to form a fundamental building block
regarding the long-time dynamics of the physical experiment.  It is thus reasonable to expect that any reasonable mathematical model describing these physical experiments
should admit oscillatory wave forms that are persistent (i.e. stable) when subject to slow wave modulations.  
Motivated by these observations, in this paper we aim at studying the rigorous modulational, i.e. side-band, stability
of periodic traveling wave forms in the conduit equation \eqref{e:conduit1}.

\begin{figure}
\begin{center}
(a) \hspace{-2em}\includegraphics[scale=0.45]{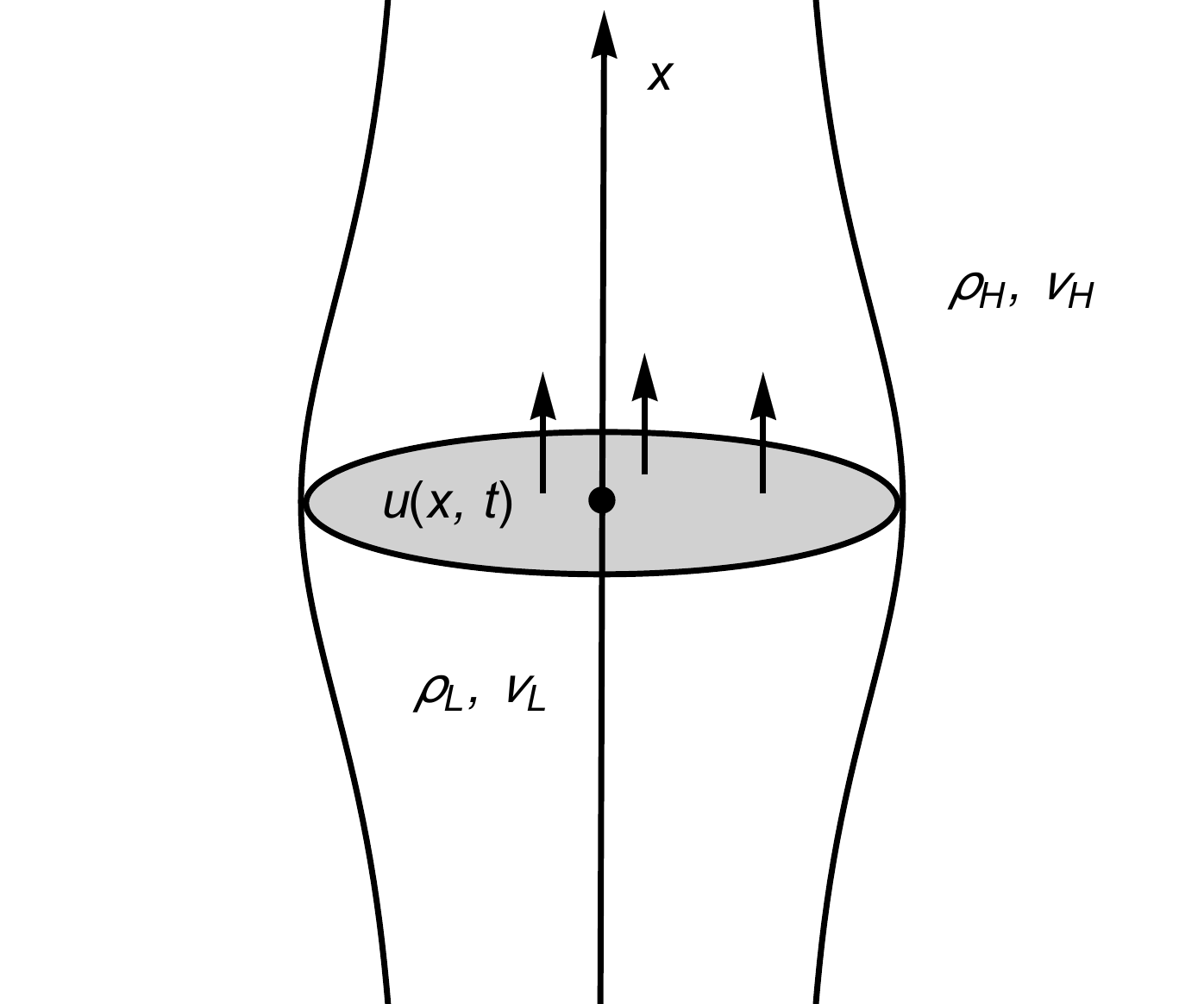}~~(b) \includegraphics[scale=0.9]{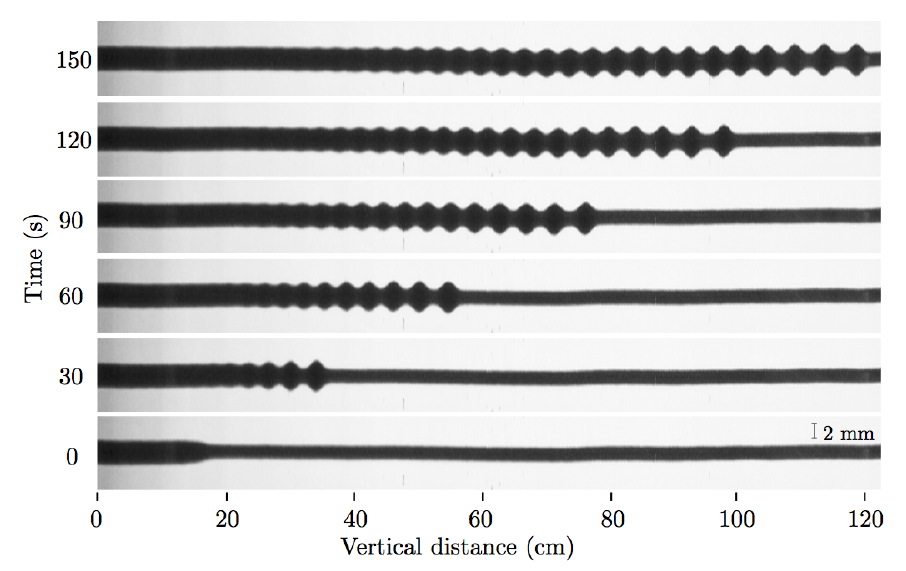}
\caption{(a) A schematic drawing for the conduit equation.  In the physical system, denoting the densities and viscosities of the heavy (outer) and light (inner) 
fluids as $\rho_H,\nu_H$ and $\rho_L,\nu_L$, respectively, the conduit equation holds under the assumption that $\rho_H>\rho_L$ and $\nu_H\gg\nu_L$.
The arrows represent rising due to buoyancy.  
(b) The formation of periodic wave trains propagating in a physical experiment.  Reprinted figure with permission from \cite{MLASH16} Copyright (2019) by the American Physical Society.
}\label{f:Schematic}
\end{center}
\end{figure}

While the stability and dynamics of solitary waves of the magma and conduit equations has been studied extensively, as described above, a rigorous
analysis of the local dynamics of periodic traveling waves seems lacking.  Note this problem is complicated by the fact that while these equations are dispersive,
they generally lack a Hamiltonian structure\footnote{We note the magma equations \eqref{e:magma} admit a Hamiltonian formulation only when $n+m=0$, which
is outside the parameter regime relevant to magma dynamics.  See \cite{SWR08} for a nonlinear stability analysis in this seemingly nonphysical case.}.
Most existing analyses seem to appeal to Whitham's theory of wave modulations: see,
for example, \cite{EKK94,LH13,MH16,MS12}.  This theory proceeds by rewriting the governing PDE in slow coordinates $(X,S)=(\eps x,\eps t)$  then uses a 
multiple scale (WKB) approximation of the solution and seeks a homogenized system, known as the Whitham modulation equations, 
describing the mean behavior of the resulting approximation.
This approach is widely used to describe the behavior of modulated periodic waves on large space and time scales, 
and in particular, is expected to predict the stability of periodic wavetrains to slow modulations.
Specifically, hyperbolicity (i.e. local well-posedness) of the Whitham system about a periodic solution $\phi$ of the governing PDE is expected to 
be a necessary condition for the stability to slow modulations of $\phi$: see, for example, \cite{W74}.

Whitham modulation theory has recently been applied to the conduit equation \eqref{e:conduit1}, where the authors,
by coupling their analysis to numerical time evolution studies and numerical analysis of the Whitham system,
identify an amplitude-dependent region of parameter space where such periodic wave trains are expected to be stable
to slow modulations.  However, we note the formal asymptotic methods used in Whitham theory are not, in general rigorously justified, thus removing a layer of rigor that would
otherwise support their predictions.  
The primary goal of this paper is to (rigorously) connect the predictions from Whitham modulation theory to the rigorous dynamical stability of the underlying
periodic wave train solutions of the conduit equation \eqref{e:conduit1}.  Specifically, our main result, Theorem \ref{T:main}, establishes that hyperbolicity of the Whitham modulation
equations about a periodic wave $\phi$ of \eqref{e:conduit1} is indeed a necessary condition for the stability of $\phi$ to slow modulations. This will be accomplished
in Section \ref{S:rigorous} by performing a rigorous analysis of the spectrum of the linearization of \eqref{e:conduit1} about such periodic traveling waves $\phi$.  Specifically,
using Floquet-Bloch theory and spectral perturbation theory we show that the spectrum near the origin  of the linearization of \eqref{e:conduit1} about $\phi$
consists of three $C^1$ curves which, locally, satisfy
\[
\lambda_j(\xi)=i\alpha_j\xi+o(\xi)
\]
where the $\alpha_j$ are precisely the characteristic speeds associated with the Whitham modulation equations about $\phi$.  Consequently, a necessary condition for
spectral stability of $\phi$ is that all the $\alpha_j$ are real, which is equivalent to the associated Whitham system being weakly hyperbolic at $\phi$.  Such a rigorous
connection between the stability of periodic waves and the Whitham modulation equations has been established previously in a number of contexts: see,
for example, \cite{BrJ10,BrJZ,BrHJ16,BNR14,Serre05} and references therein.  The specific approach here follows more closely the analysis in \cite{BNR14}, which
is based off the work of Serre in \cite{Serre05}.  

\

The organization of the paper is as follows.  In Section \ref{S:basic} we begin by recalling some basic facts about the conduit equation \eqref{e:conduit1}.  Specifically,
we discuss the conservation laws associated to \eqref{e:conduit1} as well as the existence analysis for periodic traveling wave solutions.  In Section \ref{S:whithamderivation} we
derive the Whitham moduation equations associated with \eqref{e:conduit1} and state our main result Theorem \ref{T:main}.  
We begin the proof of Theorem \ref{T:main} in Section \ref{S:rigorous}, where we perform a rigorous spectral stability calculation using spectral perturbation theory. Specifically,
there we derive a $3\times 3$ matrix which rigorously encodes the spectrum of the linearized operator associated with \eqref{e:conduit1} about a periodic traveling wave near the origin
in the spectral plane.  The proof of our main result, providing a rigorous connection between the Whitham modulation equations and the rigorous spectral analysis in Section \ref{S:rigorous}, 
is then given in  Section \ref{S:comparison}.  Finally, we end by analyzing our results for waves with asymptotically small oscillations in Section \ref{S:small}.

\

\noindent
{\bf Acknowledgments:} The authors would like to thank Mark Hoefer for several helpful orienting discussions regarding the dynamics
of viscous fluid conduits.  The work of both authors was partially supported by the NSF under grant DMS-1614895.

\section{Basic Properties of the Conduit Equation}\label{S:basic}

In this section, we collect some important basic facts about the conduit equation \eqref{e:conduit1}. 

\subsection{Conservation Laws \& Conserved Quantities} 

First, we note that it is shown in \cite[Corollary 5.7]{SSW07} that the conduit equation is globally well-posed for initial data $u(x,0)-1\in H^1(\RM)$, so long as $u(x,0)$ 
satisfies the physically reasonable requirement of being bounded away from zero.
Further, even though the conduit equation admits nearly elastic solitonic collisions, it is shown through the failure of the Painlev\'e test to not be completely integrable \cite{HC06}.
Nevertheless, \eqref{e:conduit1} admits (at least) the following two conservation laws:
\begin{equation}\label{e:conservation}
\left\{\begin{aligned}
	&u_t + (u^2-u^2(u^{-1}u_t)_x)_x=0\\
	&(u^{-1}+u^{-2}(u_x)^2)_t + (u^{-1}u_{xt} - u^{-2}u_x u_t - 2\ln|u|)_x =0
	\end{aligned}\right.
\end{equation}
Notice  \eqref{e:conservation}(i) is simply a restatement of \eqref{e:conduit1}, showing that the conduit equation itself corresponds to 
conservation of mass.  The existence of more conservation laws for the general magma equations \eqref{e:magma} was studied by Harris in \cite{H96}.
There it was shown that \eqref{e:magma} generally only admits two conservation laws\footnote{ One conservation law is always the 
magma equation \eqref{e:magma} itself, while the structure for the other law varies depending on the parameters $(n,m)$.}.  
However, this analysis was shown to be inconclusive in a couple of cases, one of which occurs when $m=1$, $n\neq 0$, which 
the conduit equation \eqref{e:conduit1} clearly falls into\footnote{The other case is given by $m = n + 1$, $n\neq 0$, where a third conservation is shown to exist.}.  
Consequently, it seems to be currently unknown if \eqref{e:conduit1} admits more conservation laws or not, though Harris seems to think 
the existence of additional conservation laws unlikely.  

Restricting to solutions that are $T$-periodic in the spatial variable $x$, the conservation laws \eqref{e:conservation} give rise to the following two conserved quantities:
\begin{equation}\label{e:mass}
M(u) := \int_0^T u\, dx,~~~Q(u) := \int_0^T \frac{u + (u_x)^2}{u^2}\, dx.
\end{equation}
As we will see, these conserved quantities will play an important part in our forthcoming analysis.  
Note that the conserved quantity $M$ corresponds to conservation of mass, while the conservation of $Q$ does not seem to have a clear physical meaning \cite{LH13}.

\subsection{Existence of Periodic Traveling Waves}\label{S:exist}

Traveling wave solutions of \eqref{e:conduit1} correspond to solutions of the form $u(x,t) = \phi(x-ct)$ for some wave profile $\phi(\cdot)$
and wave speed $c>0$.  The profile $\phi(z)$ is readily seen to be a stationary solution of the evolutionary equation
\begin{equation}\label{e:travel_conduit}
u_t-cu_z+(u^2)_z-(u^2(u^{-1}(u_t-cu_z))_z)_z=0,
\end{equation}
written here in the traveling coordinate $z=x-ct$.  
After a single integration, stationary solutions of \eqref{e:travel_conduit} are seen to satisfy the second-order ODE
\begin{equation}\label{e:profile}
-2cE = -c\phi + \phi^2 + c\phi^2(\phi^{-1}\phi')',
\end{equation}
where here $'$ denotes differentiation with respect to $z$  and $E\in\RM$ is a constant
of integration.  Multiplying \eqref{e:profile} by $\phi^{-3}\phi'$,
the profile equation can be rewritten\footnote{Alternatively, one can use the identity $2(\phi^{-1}\phi')' = \phi(\phi')^{-1}\left(\phi^{-2}(\phi')^2\right)'$.} as
\[
-2cE\phi^{-3}\phi ' = -c\phi^{-2} \phi ' + \phi^{-1} \phi ' + \tfrac{c}{2}\left[\phi^{-2}(\phi')^2\right]'
\]
and hence may be reduced by quadrature to
\begin{equation}\label{e:quad}
\tfrac{1}{2}(\phi')^2 = E - \left(\tfrac{1}{c}\phi^2\ln|\phi| + a\phi^2 + \phi\right),
\end{equation}
where here $a\in\RM$ is a second constant of integration.  By standard phase plane analysis, the existence and non-existence of bounded solutions of \eqref{e:profile}
is determined entirely by the effective potential 
\[
V(\phi;a,c):=\tfrac{1}{c}\phi^2\ln|\phi| + a\phi^2 + \phi.
\]
Indeed, for a given $a,c\in\RM$, a necessary and sufficient condition for the existence of periodic solutions of \eqref{e:quad} is that $V(\cdot;a,c)$ has a local minima.
Furthermore, since (in the physical modeling) the dependent variable $\phi$ represents the cross-sectional area of the viscous fluid conduit, we additionally
require that the local minima occur for $\phi>0$.  

To characterize the parameters $(a,c)$ for which $V(\cdot;a,c)$ has a strictly positive local minima, we study the
critical points of $V(\cdot;a,c)$.  
Noting that
\[
V_\phi(\phi;a,c)=\tfrac{2}{c}\phi \ln|\phi| + \tfrac{1}{c}\phi + 2a\phi + 1,\quad 
V_{\phi \phi}(\phi;a,c) = \tfrac{2}{c}\left(\ln|\phi| + \tfrac{3}{2} + ac  \right)
\]
and
\[
V_{\phi\phi\phi}(\phi;a,c)=\frac{2}{c\phi},
\]
we see that, since $c>0$, the derivative $V_\phi(\cdot;a,c)$ has a local maximum at $\phi_-:=-e^{-(ac+3/2)}$ and a local minimum at 
$\phi_+:=e^{-(ac+3/2)}$.  Since $V_{\phi}(\cdot;a,c)$ additionally satisfies
\[
\lim_{\phi\to0^+}V_{\phi}(\phi;a,c)=1, \quad V_\phi(\phi_+;a,c)=1-\tfrac{2}{c}e^{-(ac+3/2)}<1, \quad {\rm and}~~~\lim_{\phi\to+\infty}V_{\phi}(\phi;a,c)=+\infty
\]
for all $a\in\R$, $c>0$, it follows that the number of positive roots of $V_\phi(\cdot;a,c)$ for $c>0$ is determined by the sign of the quantity $a-\zeta(c)$, where here
\[
\zeta(c) := \frac{1}{c}\ln\left(\frac{2}{c}\right) - \frac{3}{2c}. 
\]
See Figure \ref{f:Bifurcation}(a) for a plot of $a=\zeta(c)$. 
Indeed, if $a>\zeta(c)$, then $V_\phi(\phi_+;a,c)$ is positive and hence $V_\phi(\cdot;a,c)$ has no positive roots, while $a<\zeta(c)$ implies
$V_\phi(\phi_+;a,c)$ is negative and hence $V_\phi(\cdot;a,c)$ has exactly two positive roots $0<\phi_1<\phi_2$.  In the latter case, it is clear from
the above analysis that $\phi_1$ and $\phi_2$ are local maxima and minima, respectively, of the effective potential\footnote{In the 
border case $a = \zeta(c)$, it follows that $V_{\phi}(\phi_+;a,c) = 0$.  Using the above analysis, we may then conclude $\phi_+$ is a saddle 
point of the effective potential $V(\cdot;a,c)$.} $V(\cdot;a,c)$: see Figure \ref{f:Bifurcation}(b).

\begin{figure}[t]
\begin{center}
(a)~\includegraphics[scale=0.55]{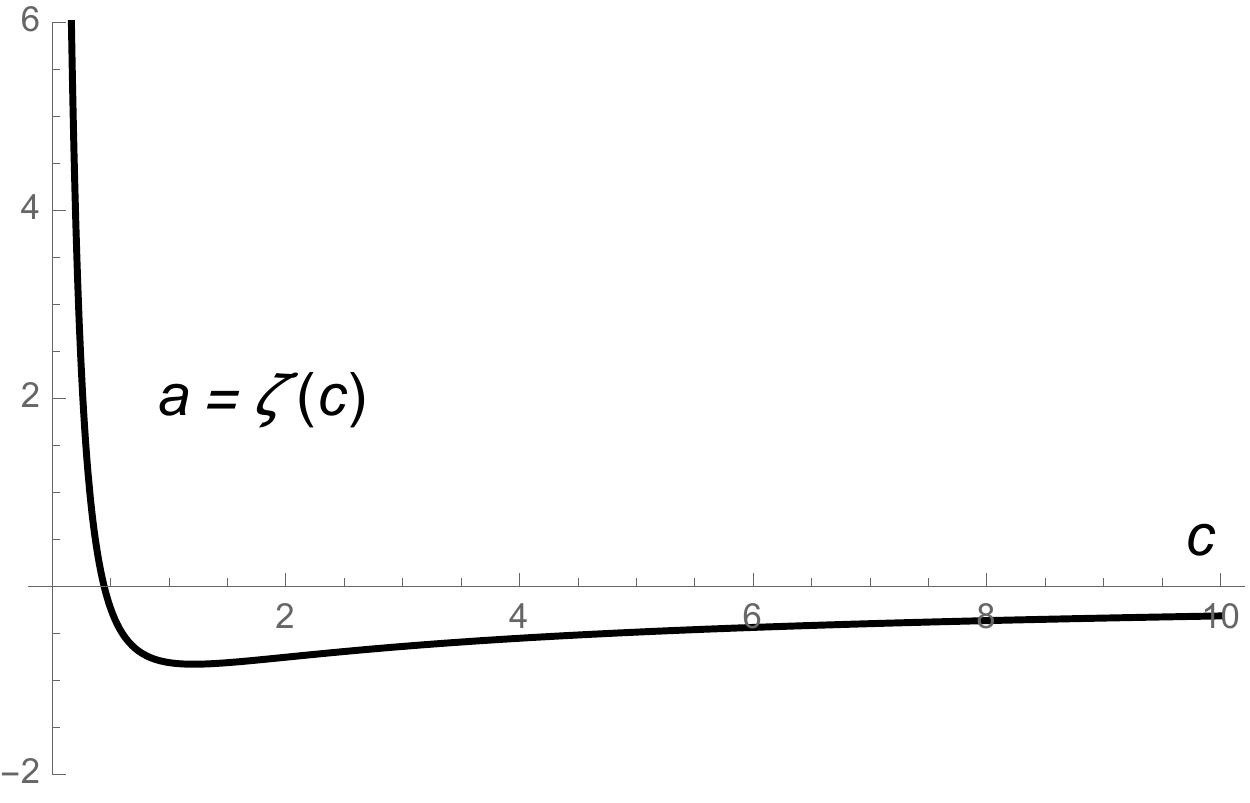}~~(b)\includegraphics[scale=0.55]{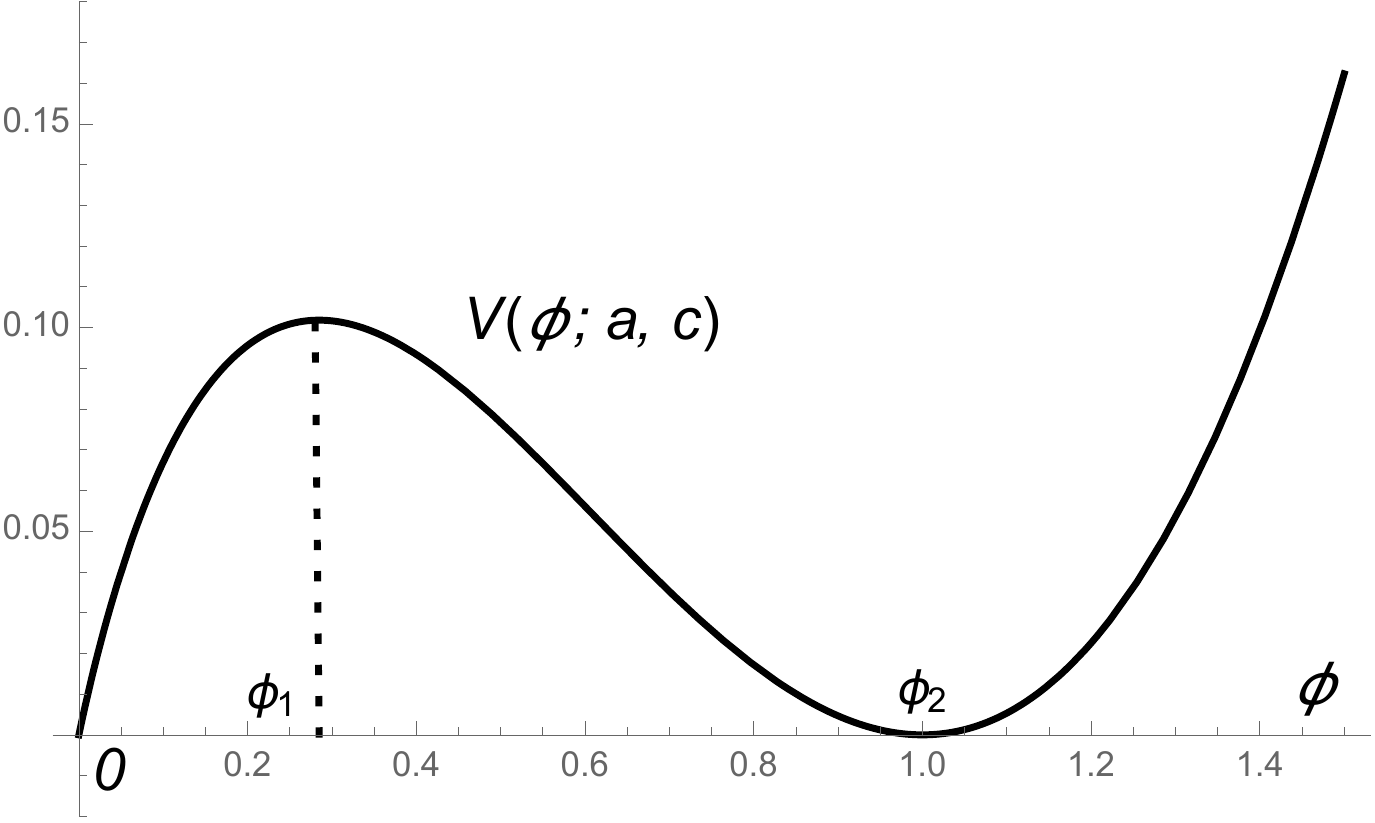}
\end{center}
\caption{(a) A plot of $\zeta(c)$ for $c>0$.  Strictly positive periodic traveling wave solutions of \eqref{e:profile} exist for $a<\zeta(c)$.  
(b)  A plot of the effective potential $V(\phi;a,c)$ for $c=1$ and $a=-1<\zeta(1)$.}
\label{f:Bifurcation}
\end{figure}

\begin{remark}
We collect here some easily verifiable properties of the function $\zeta(c)$.  First, $\zeta(c)$ only has one root, which occurs at $c=2e^{-\frac{3}{2}}$,
and one critical point (an absolute minimum), which occurs at $c=2e^{-1/2}$.  Furthermore, $\lim_{c\to 0^+}\zeta(c) = +\infty$ and $\lim_{c\to+\infty}\zeta(c) = 0$.
See Figure \ref{f:Bifurcation}(a) for a numerical plot.
\end{remark}

Returning to \eqref{e:quad}, it follows that if we define the set
\[
\mathcal{B}:=\left\{(a,E,c)\in\RM^3:c>0,~~a<\zeta(c),~~E\in(V(\phi_2(a,c);a,c),V(\phi_1(a,c);a,c))\right\},
\]
then for each $(a,E,c)\in\mathcal{B}$ the profile equation \eqref{e:profile} admits a one-parameter family, parameterized by translation invariance, of strictly positive periodic 
solutions $\phi(x;a,E,c)$ with period
\[
T=T(a,E,c)=\sqrt{2}\int_{\phi_{\rm min}}^{\phi_{\rm max}}\frac{d\phi}{\sqrt{E-V(\phi;a,c)}}=\frac{\sqrt{2}}{2}\oint_\gamma\frac{d\phi}{\sqrt{E-V(\phi;a,c)}},
\]
where here $\phi_{\rm min}\in(\phi_1,\phi_2)$ and $\phi_{\rm max}\in(\phi_2,\infty)$ are roots of $E-V(\cdot;a,c)$
corresponding to the minimum and maximum values of the profile $\phi$, respectively, 
and integration over $\gamma$ represents a complete integration from $\phi_{\rm min}$ to $\phi_{\rm max}$, and then back to $\phi_{\rm min}$ again.
Naturally, one must appropriately choose the branch of the square root in each direction.  Alternatively, one could interpret the contour $\gamma$ as a closed (Jordan) curve
in the complex plane $\CM$ that encloses a bounded set containing both $\phi_{\rm min}$ and $\phi_{\rm max}$.  Since the values $\phi_{\rm min/max}$ are smooth functions
of the traveling wave parameters $(a,E,c)$, a standard procedure shows the above integrals may be regularized at the square root branch points and hence represent $C^1$ functions
of $a$, $E$, and $c$.   In this way, we have proven the existence of a $4$-parameter family (in fact, a $C^1$ manifold) of periodic traveling wave solutions of \eqref{e:profile}:
\[
\phi(x-ct+x_0;a,E,c),~~x_0\in\RM,~~(a,E,c)\in\mathcal{B}
\]
with period $T=T(a,E,c)$.  Notice as $E\searrow V(\phi_2;a,c)$ the profile $\phi(\cdot;a,E,c)$ converges to a constant solution, while $T(a,E,c)\to+\infty$ 
as $E\nearrow V(\phi_1,a,c)$ which corresponds to a solitary wave limit.  Without loss of generality, we may choose $x_0$ such that $\phi$ is an even function, 
which we will do throughout the rest of the paper.

\begin{remark}\label{r:solitary_exist}
In \cite{SW08}, the authors study the existence and nonlinear stability of traveling solitary waves of the more general class of magma equations \eqref{e:magma}.
In the context of the conduit equation \eqref{e:conduit1}, the authors considered solitary waves satisfying $\phi \to 1$ as $|x|\to\infty$, which they show to exist
for all $c>2$.  Taking $|x|\to\infty$ in \eqref{e:profile} and requiring that $\phi=1$ be a local maximum of $V(\cdot;a,c)$, 
we see  these waves correspond to the choice $2cE = c-1$, $a=E-1$, and $c>2$.  See also Remark \ref{r:solitary_lin} below.
\end{remark}

By a similar procedure as above, the conserved quantities $M$ and $Q$ defined in \eqref{e:mass} can be restricted to the manifold of periodic traveling wave solutions
of \eqref{e:conduit1}.  Indeed, given a $T$-periodic traveling wave $\phi(\cdot;a,E,c)$ of \eqref{e:travel_conduit}, we can define the functions
$M,Q:\mathcal{B}\to\RM$ via
\begin{equation}\label{e:MQ}
\left\{\begin{aligned}
M(a,E,c)& := \int_0^{T(a,E,c)} \phi(z;a,E,c)\, dz=\frac{\sqrt{2}}{2}\oint_\gamma\frac{\phi d\phi}{\sqrt{E-V(\phi;a,c)}}\\
Q(a,E,c) &:= \int_0^{T(a,E,c)} \frac{\phi(z;a,E,c) + (\phi '(z;a,E,c))^2}{(\phi(z;a,E,c))^2}\, dz=\frac{\sqrt{2}}{2}\oint_\gamma\frac{\left(\phi+2(E-V(\phi;a,c))\right)d\phi}{\phi^2\sqrt{E-V(\phi;a,c)}}
\end{aligned}\right.
\end{equation}
where the contour integral over $\gamma$ is defined as before.  Following previous arguments, the above integrals can be regularized near their square root singularities and hence
represent $C^1$ functions on $\mathcal{B}$.  
As we will see, the gradients of these conserved quantities along the manifold of periodic traveling wave solutions of 
\eqref{e:conduit1} will play an important role in our analysis.

\section{The Whitham Modulation Equations}\label{S:whithamderivation}

In this section, we begin our study of the long-time dynamics of an arbitrary amplitude, slowly modulated periodic
traveling wave solution of the conduit equation \eqref{e:conduit1}.
An often used, yet completely formal, approach to study the dynamics of  such slowly modulated periodic waves  is to analyze the associated Whitham
modulation equations \cite{W74}.  While Whitham originally formulated this approach in terms of averaged conservation
laws\cite{W65}, it was later shown to be equivalent to an asymptotic reduction derived through formal multiple-scales (WKB) expansions \cite{L66}. 
For completeness, we recall the derivation in the context of the conduit equation \eqref{e:conduit1}: see also the description in \cite[Appendix C]{MH16}.

To provide an asymptotic description of the slow modulation of periodic traveling wave solutions of \eqref{e:conduit1}, we separate both space and time
into separate fast and slow scales.  For $\eps>0$ sufficiently small, we introduce the ``slow" variables $(X,S):=(\eps x,\eps t)$ and note that, in the slow coordinates,
\eqref{e:conduit1} can be written as
\begin{equation}
\eps u_S + 2\eps u u_X - \eps^3 u u_{XXS} + \eps^3 u_{XX} u_S = 0\label{e:conduitslow}
\end{equation}
Following Whitham \cite{W65,W74}, we seek solutions of \eqref{e:conduitslow} of the form
\[
u(X,S;\eps) = u^0 \left(X,S,\tfrac{1}{\eps}\psi(X,S)\right) + \eps u^1 \left(X,S,\tfrac{1}{\eps}\psi(X,S)\right) + O(\eps^2)
\]
where here the phase $\psi$ is chosen to ensure that the functions $u^j$ are $1$-periodic functions of the third coordinate $\theta=\psi(X,S)/\eps$.
Substituting this ansatz into \eqref{e:conduitslow} yields a hierarchy  of equations in algebraic orders of $\eps$ that must all be simultaneously satisfied.  
At the lowest order in $\eps$, which here corresponds to $\mathcal{O}(1)$, we find the relation
\begin{equation}\label{e:w1}
\psi_S u_\th^0 + \psi_X ((u^0)^2)_\th - \psi_X^2\psi_S ((u^0)^2((u^0)^{-1}u^0_\th)_\th)_\th = 0.
\end{equation}
After the identification $k := \psi_X$ and $\o := -\psi_S$ as the spatial and temporal frequencies of the modulation, respectively, and 
$c := \frac{\o}{k}$ as the wave speed, \eqref{e:w1} is recognized, up to a global factor of $-k$, as the derivative with respect to the ``fast'' variable $\theta$ of the nonlinear profile equation \eqref{e:profile}, 
rescaled for $1$-periodic functions.  
Note that, here, $k$, $\omega$ and $c$ are now functions of the slow variables $X$ and $S$.  Consequently, for a fixed $X$ and $S$ we may choose 
$u^0$ to be a periodic traveling wave solution of \eqref{e:conduit1}, and hence of the form
\[
u^0(\theta,X,S)=\phi(\theta,a(X,S), E(X,S), c(X,S))
\]
for some even solution $\phi$ of \eqref{e:profile} with $(a(X,S),E(X,S),c(X,S))\in\mathcal{B}$.  Notice the consistency condition $(\psi_X)_S=(\psi_S)_X$ implies the 
local wave number $k$ and wave speed $c$ must slowly evolve according to the relation
\begin{equation}\label{e:nonlineardispersion}
k_S+(k c)_X=0,
\end{equation}
which is sometimes referred to as ``conservation of waves".
Note that \eqref{e:nonlineardispersion} effectively serves as a nonlinear dispersion relation.  Indeed, in the case of linear waves one would have $\psi(X,S)=k(X-cS)$, 
which clearly satisfies \eqref{e:nonlineardispersion}.

Continuing to study the above hierarchy of equations, at $\mathcal{O}(\eps)$ we find
\begin{equation}\label{e:forcing}
k\partial_\theta \mathcal L[u^0]u^1=G[u^0]u^0_S+F(u^0),
\end{equation}
where here $\mathcal L[u^0]$ and $G[u^0]$ are linear differential operators defined via
\[
\left\{\begin{aligned}
\mathcal L[u^0]&:= \left(c  -2 u^0  -k^2 c u_{\th\th}^0\right)  + 2k^2 c u_\th^0\partial_\theta - k^2 c u^0 \partial_{\th}^2,\\
G[u^0]&:=\left(1+k^2u_{\theta\theta}^0\right)-k^2u^0\partial_{\theta}^2
\end{aligned}\right.
\]
supplemented with $1$-periodic boundary conditions, and
\begin{align*}
F(u^0)&=((u^0)^2)_X + \frac{3}{2}(k^2)_X c u^0 u_{\th\th}^0 - \frac{1}{2}(k^2)_X c (u_\th^0)^2\\
&\quad+ 2k^2 c_X u^0 u_{\th\th}^0 + 2k^2 c u^0 u_{X\th\th}^0 - k^2 c((u_\th^0)^2)_X
\end{align*}
contains all the nonlinear terms in $u^0$ and its derivatives.  Treating \eqref{e:forcing} as a forced linear equation for the unknown $u^1$, 
it follows by the Fredholm alternative that \eqref{e:forcing} is solvable in the class of $1$-periodic functions if and only if
\[
G[u^0]u^0_S+F(u^0)\perp{\rm ker}_{\rm L^2_{\rm per}(0,1)}\left(\mathcal L^\dag[u^0]\partial_\theta \right),
\]
where here 
\[
\mathcal L^\dag[u^0]=\left(c  -2 u^0  -k^2 c u_{\th\th}^0\right)  - 2k^2 c \partial_\theta\left(u_\th^0\;\cdot\,\right) - k^2 c \partial_{\theta}^2\left(u^0\;\cdot\,\right)
\]
denotes the adjoint operator of $\mathcal L[u^0]$ on $L^2_{\rm per}(0,1)$.  In particular, noting the identity\footnote{Although you can verify the identity using the forms listed above, this identity  immediately follows from the alternative form for $\mathcal L[f]$ given in \eqref{L_rep2} below.}  
\begin{equation}\label{e:L_relationship}
\mathcal L^\dag[f]\left(f^{-3}g\right)=f^{-3}\mathcal L[f]g,
\end{equation}
a straightforward calculation\footnote{Here, we are using that differentiating \eqref{e:w1} with respect to $\theta$ gives $\mathcal L[u^0]u^0_\theta=0$.  Further, we note the third
linearly independent solution of $\mathcal L^\dag[u_0]v_\theta=0$ is not periodic in $\theta$.}  shows that 
\[
{\rm ker}_{L^2_{\rm per}(0,1)}\left(\mathcal L^\dag[u^0]\partial_\theta \right)={\rm span}\left\{1,(u^0)^{-2}\right\}.
\]
Thus, our two solvability conditions become
\[
\left<1,G[u^0]u^0_S+F(u^0)\right>_{L^2_{\rm per}(0,1;d\theta)}=0\quad{\rm and}\quad \left<(u^0)^{-2},G[u^0]u^0_S+F(u^0)\right>_{L^2_{\rm per}(0,1;d\theta)}=0.
\]

To put the above solvability conditions in a more useful form, we note that since 
\[
G^\dag[u^0] = \left(1+k^2u_{\th\th}^0\right) - k^2 \d_{\th}^2\left(u^0 \cdot\right),
\] 
integration by parts (in $\theta)$ and the identity $G^\dag[u_0]1=1$ imply that the first equation above can be rewritten as
\begin{align*}
\partial_SM(u^0)
&= - \LA 1, ((u^0)^2)_X - 2(k^2)_X c (u_\th^0)^2 - 2k^2 c_X (u_\th^0)^2 - 2k^2 c((u_\th^0)^2)_X\RA\\
&= \partial_X\LA 1, 2k^2 c (u_\th^0)^2 - (u^0)^2\RA_{L^2_{\rm per}(0,1;d\theta)},
\end{align*}
where here $M(u^0)=\int_0^1u^0(\theta)d\theta$ is simply the conserved quantity $M$ in \eqref{e:mass} evaluated at the $1$-periodic traveling wave $u^0(\cdot)$.  
Similarly, using the identities
\[
\LA (u^0)^{-2}, u^0 u_{\th\th}^0 \RA = \LA (u^0)^{-2}, (u_\th^0)^2\RA \quad \LA (u^0)^{-2}, u^0 u_{X\th\th}^0\RA = \LA (u^0)^{-2}, u_\th^0 u_{X\th}^0\RA,
\] 
which follow from integration by parts, the second solvability condition can be rewritten as
\begin{align*}
\left<G^\dag[u^0](u^0)^{-2},u^0_S\right>&=-\LA (u^0)^{-2}, ((u^0)^2)_X + \frac{3}{2}(k^2)_X c (u_\th^0)^2 - \frac{1}{2}(k^2)_X c (u_\th^0)^2\RA\\
&\qquad -\LA  (u^0)^{-2},2k^2 c_X (u_\th^0)^2 + 2k^2 c u_\th^0 u_{X\th}^0 - k^2 c((u_\th^0)^2)_X\RA\\
&=-\left<(u^0)^{-2},((u^0)^2)_X\right>-2k\left(kc\right)_X\left<(u^0)^{-2},(u^0_\theta)^2\right>.
\end{align*}
Using \eqref{e:nonlineardispersion} and the fact that
\[
\partial_SQ(u^0)=-\left<G^\dag[u^0](u^0)^{-2},u_S\right>+2kk_S\left<(u^0)^{-2},(u^0_\theta)^2\right>,
\]
where here $Q(u^0)=\int_0^1\frac{u^0+(u^0_\theta)^2}{(u^0)^2}d\theta$ is the conserved quantity $Q$ in \eqref{e:mass} evaluated at the $1$-periodic 
traveling wave $u^0(\cdot)$, the above can be rewritten as
\[
\partial_SQ(u^0)=\left<(u^0)^{-2},((u^0)^2)_X\right>=\partial_X\left<1,2\ln|u^0|\right>.
\]
Taken together, \eqref{e:nonlineardispersion} and the above solvability conditions yield the first order, $3\times 3$ system
\begin{equation}\label{e:whitham}
\left\{ \begin{aligned} 	k_S &+ \partial_X(kc)=0\\
					M_S &= \partial_X\LA 1, 2k^2 c (u_\th^0)^2 - (u^0)^2\RA_{L^2_{\rm per}(0,1)}\\
					Q_S &= \partial_X\LA 1, 2\ln|u^0|\RA_{L^2_{\rm per}(0,1)}
\end{aligned}\right.
\end{equation}
which, by the above formal arguments, is expected  to govern (at least to leading order) the slow evolution of the wave number $k$ and the conserved quantities $M$ and $Q$
of a slow modulation of the periodic traveling wave $u^0$.

System \eqref{e:whitham} is referred to as the Whitham modulation system associated to the conduit equation \eqref{e:conduit1}.  
Heuristically, it is expected that the Whitham modulation equations \eqref{e:whitham} relate to the dynamical stability of periodic traveling wave
solutions of \eqref{e:conduit1} in the following way.  Suppose $(a_0,E_0,c_0)\in\mathcal{B}$ so that $\phi(\cdot;a_0,E_0,c_0)$ is an even, $T=1/k$-periodic 
solution of \eqref{e:profile}.  From the above formal analysis, we see that \eqref{e:profile} has a modulated periodic traveling wave of the form
\[
u(x,t;\eps)=\phi\left(\tfrac{1}{\eps}\psi(\eps x,\eps t);a(\eps x,\eps t),E(\eps x,\eps t),c(\eps x,\eps t)\right)+o(\eps),
\]
where the parameters $(a(\eps x,\eps t),E(\eps x,\eps t),c(\eps x,\eps t))$ evolve near $(a_0,E_0,c_0)$ in $\mathcal{B}$ in such a way
that $k$, $M$, and $Q$ satisfy the Whitham system \eqref{e:whitham}.  Note this requires that the nonlinear mapping
\[
\RM^3\ni (a,E,c)\mapsto (k,M,Q)\in\RM^3
\]
be locally invertible near $(a_0,E_0,c_0)$.  By the implicit function theorem, it is sufficient to assume the Jacobian of this mapping at $(a_0,E_0,c_0)$ is non-zero\footnote{This condition
will appear later in our rigorous theory as well: see the discussion following the proof of Theorem \ref{T:gker} below.  It also appears in the formal work
of Maiden \& Hoefer \cite{MH16}.}.
In particular, any $1$-periodic solution $\phi_0$
of \eqref{e:profile}, being independent of the slow variables $X$ and $S$, is necessarily a constant solution of \eqref{e:whitham}.
The stability of $\phi_0$ to slow modulations is thus expected to be governed (to leading order, at least) by the linearization of \eqref{e:whitham}
about $\phi_0$.  Specifically, using the chain rule to rewrite \eqref{e:whitham} in the quasilinear form
\begin{equation}\label{e:whitham_quasilinear}
\left(\begin{array}{c}k\\M \\ Q\end{array}\right)_S= {\bf D}(u^0)\left(\begin{array}{c}k \\M \\ Q\end{array}\right)_X,
\end{equation}
where here 
\[
{\bf D}(u)=\left(\begin{array}{ccc}
								- (c  + k c_k) & - k c_M & - k c_Q \\
								 \LA 1, 2k^2 c u_\th^2 - u^2\RA_k & \LA 1, 2k^2 c u_\th^2 - u^2\RA_M & \LA 1, 2k^2 c u_\th^2 - u^2\RA_Q \\
								\LA 1, 2\ln|u|\RA_k & \LA 1, 2\ln|u|\RA_M & \LA 1, 2\ln|u|\RA_Q
							\end{array}\right),
\]
it is natural to expect the stability of $\phi_0$ to slow modulations to be governed by the eigenvalues of the $3\times 3$ matrix ${\bf D}(\phi_0)$.
Indeed, linearizing \eqref{e:whitham_quasilinear} about the constant solution $\phi_0$, we see that the eigenvalues of the linearization are of the form
\[
\widetilde{\lambda_j}(\xi)=i\xi\alpha_j,
\]
where $\{\alpha_j\}_{j=1}^3$ are the eigenvalues of ${\bf D}(\phi_0)$ and $\xi\in\RM$.  
Consequently, if the Whitham system is weakly hyperbolic\footnote{Note full hyperbolicity of the system additionally requires the eigenvalues
$\alpha_j$ are semi-simple, i.e. that their algebraic and geometric multiplicities agree.}  at $\phi_0$, so that the eigenvalues of ${\bf D}(\phi_0)$ are all real, then 
the eigenvalues of the linearization of \eqref{e:whitham} are purely imaginary, indicating  a marginal (spectral) stability.  
Conversely, if ${\bf D}(\phi_0)$ has an eigenvalue with non-zero imaginary part, in which case \eqref{e:whitham_quasilinear} is elliptic at $\phi_0$,
then the linearization of \eqref{e:whitham}  has eigenvalues with positive real part, indicating (spectral) instability of $\phi_0$.

\

The goal of this paper is to rigorously validate the above predictions of Whitham modulation theory as they pertain to the stability of periodic traveling
wave solutions of \eqref{e:conduit1}.  Following the works in \cite{BrJZ,BrHJ16,BNR14}, this will be accomplished by using rigorous spectral perturbation
theory to analyze the spectrum of the linearization of \eqref{e:travel_conduit} about such a solution and, in particular, relating the spectrum
of the linearization in a neighborhood of the origin in the spectral plane to the eigenvalues of the matrix ${\bf D}(\phi_0)$ defined above.  
Our main result is the following, which
establishes that weak-hyperbolicity of the Whitham system \eqref{e:whitham} is indeed a necessary condition for the spectral stability of the underlying
wave $\phi_0$.  

\begin{theorem}\label{T:main}
Suppose $\phi_0$ is an even, $T_0=1/k_0$-periodic, strictly positive traveling wave solution of \eqref{e:conduit1} with wave speed $c_0>0$, and that the set of nearby periodic traveling
wave profiles $\phi$ with speed close to $c_0$ is a $3$-dimensional manifold parameterized 
by $(k,M(\phi),Q(\phi))$, where $1/k$ denotes the fundamental period of $\phi$.
Then a necessary condition for $\phi_0$ to be spectrally stable is that the Whitham modulation system \eqref{e:whitham}
be weakly hyperbolic at $(k_0,M(\phi_0),Q(\phi_0))$, in the sense that all their characteristic speeds must be real.
\end{theorem}

To prove Theorem \ref{T:main} we will show that, under appropriate non-degeneracy assumptions, the spectrum of the linearization of \eqref{e:travel_conduit} about
a periodic traveling wave $\phi_0$ consists, in a sufficiently small neighborhood of the origin, of precisely three $C^1$ curves which expand as
\[
\lambda_j(\xi)=\widetilde{\lambda_j}(\xi)+\mathcal{O}(\xi)=i\alpha_j\xi+\mathcal{O}(\xi),\quad |\xi|\ll 1,
\]
where here the $\alpha_j$ are precisely the eigenvalues of the matrix ${\bf D}(\phi_0)$.  Interestingly, this shows that spectral stability in a neighborhood of the origin
of $\phi_0$, otherwise known as ``modulational stability", cannot be concluded from the weak, or even strong, hyperbolicity of the Whitham modulation system \eqref{e:whitham}.
While we do not pursue it here, such information may be able to be deduced from determining the second-order corrector in $\xi$ to the spectral curves $\lambda_j(\xi)$ deduced above.

\section{Rigorous Modulational Stability Theory}\label{S:rigorous} 

We now begin our rigorous mathematical study of the dynamical stability of periodic traveling wave solutions
of \eqref{e:conduit1} when subject to localized, i.e. integrable, perturbations on the line.  
Following  \cite{BrJZ,BrHJ16,BNR14}, we conduct a detailed analysis of the spectral problem associated with
the linearization of \eqref{e:conduit1} about a periodic traveling wave solution.  The first step in this analysis is to understand the 
structure of the generalized kernel of the associated linearized operators when subject to perturbations that are co-periodic with the underlying
wave.  With this information in hand, we then use Floquet-Bloch theory and rigorous spectral perturbation theory to obtain an asymptotic
description of the spectrum of the linearization considered as an operator on $L^2(\RM)$ in a sufficiently small neighborhood of the origin.

\subsection{Linearization \& Set Up}\label{S:linearize} 

To begin, let $(a,E,c)\in\mathcal{B}$ and denote by $\phi=\phi(\cdot;a,E,c)$ the corresponding even, $T=T(a,E,c)$-periodic equilibrium solution of 
\eqref{e:travel_conduit}.
We are now interested in rigorously describing the local dynamics of \eqref{e:travel_conduit} near $\phi$.  Specifically, we are interested in understanding if $\phi$ is stable
to small \emph{localized}, i.e. integrable on $\RM$, perturbations.  
To this end, we note that the linearization of \eqref{e:travel_conduit} about $\phi$ is given by
\begin{equation}\label{lin_evolve}
G[\phi]v_t=\partial_z\mathcal{L}[\phi]v,
\end{equation}
where here $\G$ and $\mathcal{L}[\phi]$ are linear operators on $L^2(\RM)$ defined by
\begin{equation}\label{ops}
\left\{\begin{aligned}
\G f& := f - (\phi^2 (\phi^{-1} f)')' = f + \phi '' f - \phi f '',\\
\mathcal{L}[\phi] f &:= cf - 2\phi f - c\phi '' f + 2c\phi ' f ' - c\phi f ''.
\end{aligned}\right.
\end{equation} 
Note that $\mathcal{L}[\phi]$ can also be written in the form
\begin{equation}\label{L_rep2}
\mathcal{L}[\phi] f = cf - 2\phi f - 2c\phi(\phi^{-1}\phi ')' f - c\phi^2 (\phi^{-1}f)''.
\end{equation}
Observe these are both closed linear operators on $L^2(\RM)$ with densely defined domains $H^2(\RM)$.  As the linear evolution equation \eqref{lin_evolve} is autonomous in time,
its dynamics can be (at least partly) understood by studying the associated generalized spectral problem
\begin{equation}\label{lin}
\lambda G[\phi]v=\partial_z\mathcal{L}[\phi]v,
\end{equation}
posed on $L^2(\RM)$, where here $\lambda\in\CM$ is a spectral parameter corresponding to the temporal frequency of the perturbation. 
To put \eqref{lin} in a more standard form, we note the following lemma.

\begin{lemma}\label{L:inverse}
The operator $\G: H^1(\R) \to L^2(\R)$ defined in \eqref{ops}(i) is a (weakly) invertible operator.  That is, for every $g\in L^2(\RM)$
the equation
\[
G[\phi]v=g
\]
has a unique weak solution in $H^1(\RM)$.
\end{lemma}

\begin{proof}
Observe that by defining
\[
H[\phi] f := \G(\phi f) =  \phi f - (\phi^2 f ')'
\]
we have $H[\phi](\phi^{-1}f) = \G f$ so that, in particular, $\G$ is (weakly) invertible if and only if $H[\phi]$ is (weakly) invertible. 
Since $\phi>0$ uniformly, it follows that $H[\phi]$ is a symmetric, uniformly elliptic differential operator.  Consequently, a standard argument
using the Riesz representation theorem implies that for every $g\in L^2(\RM)$ the elliptic equation
\[
H[\phi]f=g
\]
has a unique weak solution in $H^1(\RM)$, i.e. that $H[\phi]:H^1(\RM)\to L^2(\RM)$ is (weakly) invertible.  The result now follows.
\end{proof}

\begin{remark}
The (weak) invertibility of $\G$ implies that the bilinear form generated by\footnote{Throughout, we denote, with a slight abuse of notation, the operator $\left(G[\phi]\right)^{-1}$ by 
simply $G^{-1}[\phi]$.  The same abuse of notation will be used when referring to adjoints of operators depending on $\phi$.}
$\Ginv$ is well-defined on $L^2(\RM)$.  As we will see, this will be sufficient in order
to verify Theorem \ref{T:main}.  
\end{remark}

By Lemma \ref{L:inverse}, the generalized spectral problem \eqref{lin} is equivalent to the spectral problem for the linear operator
\begin{equation}\label{A}
A[\phi] := \Ginv \d_z \mathcal{L}[\phi]
\end{equation}
considered as a closed, densely defined linear operator on $L^2(\RM)$.  Motivated by the above considerations, we say that a periodic traveling wave $\phi$
of \eqref{e:conduit1} is said to be \emph{spectrally unstable} if the $L^2(\RM)$-spectrum of $A[\phi]$ intersects the open right half plane, i.e. if
\[
\sigma_{L^2(\RM)}\left(A[\phi]\right)\cap\left\{\lambda\in\CM:{\rm Re}(\lambda)>0\right\}\neq\emptyset,
\]
while it is \emph{spectrally stable} otherwise.  This motivates a detailed study of the spectrum of the linear operator $A[\phi]$.

\begin{remark}
Observe that while \eqref{e:conduit1} is a nonlinear dispersive PDE, it does not possess a Hamiltonian structure.  Consequently, while the spectrum of $A[\phi]$ is 
symmetric about the real axis, owing to the fact that $\phi$ is real-valued, it is not necessarily symmetric about the imaginary axis.
\end{remark}

\begin{remark}\label{r:solitary_lin}
Recall that in \cite{SW08} the authors considered the stability of solitary traveling wave solutions of the magma equations \eqref{e:magma}, which corresponds
to the conduit equation \eqref{e:conduit1} when $(n,m)=(2,1)$.  In that case, the  linearization (in
\cite{SW08}) of \eqref{e:travel_conduit} about a solitary wave $\phi$ is given, after some manipulation, by
\[
\G v_t = v_t - (\phi^2(\phi^{-1}v_t)_z)_z = \partial_z\left[-cv + 2(c-1)\phi^{-1}v  - c\phi^2\partial_z^2(\phi^{-1}v)\right]
\]
which, using \eqref{L_rep2}, is equivalent to our representation \eqref{lin} provided that  $\phi$ satisfies 
\[
\phi + c\phi(\phi^{-1}\phi')' = c + (1-c)\phi^{-1}
\]
Recalling Remark \ref{r:solitary_exist}, this latter condition follows directly from \eqref{e:profile} with the choice $2cE=c-1$ associated to the solitary
waves considered in \cite{SW08}.
\end{remark}

To begin our study of the $L^2(\RM)$-spectrum of $A[\phi]$, we first note that since $A[\phi]$ has $T$-periodic coefficients, standard results from Floquet theory
dictate that non-trivial solutions of the spectral problem
\begin{equation}\label{spec}
A[\phi]v=\lambda v
\end{equation}
cannot be integrable\footnote{In particular, such solutions can not have finite norm in $L^p(\RM)$ for any $1\leq p<\infty$.} 
on $\RM$ and that, at best, they can be bounded functions on the line: see, for example, \cite{KP_book,RS4}.  Further, any bounded solution of \eqref{spec} must be of the form
\[
v(x)=e^{i\xi x}w(x)
\]
for some $w\in L^2_{\rm per}(0,T)$ and $\xi\in[-\pi/T,\pi/T)$.  From these observations, it can be shown that $\lambda\in \CM$ belongs to the $L^2(\RM)$-spectrum of $A[\phi]$
if and only if there exists a $\xi\in[-\pi/T,\pi/T)$ such that the problem
\begin{equation}\label{bvp}
\left\{\begin{aligned}
			&A[\phi]v=\lambda v\\
			&v(x+T)=e^{i\xi T}v(x),
			\end{aligned}\right.
\end{equation}
has a non-trivial solution, or, equivalently, if and only if there exists a $\xi\in[-\pi/T,\pi/T)$ and a non-trivial $w\in L^2_{\rm per}(0,T)$
such that
\begin{equation}\label{bloch_spec}
\lambda w = e^{-i\xi x}A[\phi]e^{i\xi x}w=:A_\xi[\phi]w.
\end{equation}
For details, see \cite{KP_book,RS4,J13}, for example.
The one-parameter family of operators $A_\xi[\phi]$ are called the \emph{Bloch operators} associated to $A[\phi]$, and $\xi$ is referred to as the \emph{Bloch parameter} or sometimes
as the \emph{Bloch frequency}.
Since the Bloch operators have compactly embedded domains in $L^2_{\rm per}(0,T)$, it follows for each $\xi\in[-\pi/T,\pi/T)$ that the $L^2_{\rm per}(0,T)$
spectrum of $A_\xi[\phi]$ consists entirely of isolated eigenvalues with finite algebraic multiplicities.  Furthermore, we have the spectral decomposition
\begin{equation}\label{spec_decomp}
\sigma_{L^2(\RM)}\left(A[\phi]\right)=\bigcup_{\xi\in[-\pi/T,\pi/T)}\sigma_{L^2_{\rm per}(0,T)}\left(A_\xi[\phi]\right),
\end{equation}
thereby continuously parameterizing the essential $L^2(\RM)$-spectrum of $A[\phi]$ by a one-parameter family of $T$-periodic eigenvalues of the associated
Bloch operators.  For more details, see \cite{RS4}.

To determine the spectral stability of a periodic traveling wave $\phi$, one must therefore determine all of the $T$-periodic eigenvalues for each Bloch operator for $\xi\in[-\pi/T,\pi/T)$.
Outside of some very special cases, one does not expect to be able to do this complete spectral analysis explicitly.  Thankfully, however, for the purposes of modulational stability analysis,
we need only consider the spectrum of the operators $A_\xi[\phi]$ in a neighborhood of the origin in the spectral plane and only for $|\xi|\ll 1$.
To motivate this, observe from \eqref{bvp} that the spectrum of $A_0[\phi]$ corresponds to the spectral stability of $\phi$ to $T$-periodic perturbations, i.e. to perturbations
with the same period as the carrier wave.  Similarly, $|\xi|\ll1$ corresponds to long wavelength perturbations of the carrier wave.  Furthermore, slow modulations of $\phi$
form a special class of long wavelength perturbations in which the effect of the perturbation is to slowly vary, namely modulate, the wave characteristics -- the parameters $a$, $E$ and $c$ in
the present setting -- and the translational mode.  As we will see, variations in these parameters naturally provide spectral information about the co-periodic Bloch operator $A_0[\phi]$
at the origin in the spectral plane.  From the above considerations, it is natural to expect that the spectral stability of the underlying wave $\phi$ to slow modulations corresponds
to the case when the spectrum of the Bloch operators $A_\xi[\phi]$ near $(\lambda,\xi)=(0,0)$ lie in the closed left half-plane.  For more discussion regarding this motivation,
see \cite{BrHJ16}.

In order to prove Theorem \ref{T:main}, our program will roughly break down into three steps.  First, we will analyze the structure of the generalized kernel of the unmodulated Bloch operator
$A_0[\phi]$, showing that, under certain geometric conditions, this operator has $\lambda=0$ as an eigenvalue with algebraic multiplicity three and geometric multiplicity
two.  Secondly, we will use rigorous spectral perturbation theory to examine how the spectrum near the origin of the modulated operators $A_\xi[\phi]$ bifurcates from
$\lambda=0$ for $0<|\xi|\ll 1$.  Through this, we will derive a $3\times 3$ linear system that encodes the leading order asymptotics of the spectral curves
near $\lambda=0$ for $0<|\xi|\ll 1$.  Finally, we will see by a direct term by term comparison that this linear system, derived through rigorous spectral perturbation theory,
agrees \emph{exactly} (up to a harmless shift by the identity) with the linearized Whitham modulation system \eqref{e:whitham}.

\subsection{Analysis of the Unmodulated Operators}\label{S:gkernels}

As described above, the first main step in our analysis is to understand the $T$-periodic generalized kernel of the unmodulated operator $A_0[\phi]$ defined in \eqref{A}, as well
as its adjoint operator\footnote{Information about the adjoint is necessary to construct the spectral projections for $A_0[\phi]$ at $\lambda=0$.}  $A_0^\dag[\phi]$.  
We begin by characterizing the $T$-periodic kernel of $\mathcal{L}[\phi]$ and its adjoint.  To this end, note that differentiating
 the profile equation \eqref{e:profile} with respect to $z$, as well as with respect to the parameters
$a$, $E$, and $c$, yields the identities
\begin{equation}\label{e:L_ident1}
\mathcal{L}[\phi]\phi'=0=\mathcal{L}[\phi]\phi_a,~~\mathcal{L}[\phi]\phi_E=2c,~~\mathcal{L}[\phi]\phi_c=2E-\left(\phi+(\phi')^2-\phi\phi''\right).
\end{equation}
Recalling that $\mathcal{L}[\phi]$ and $\mathcal{L}^\dag[\phi]$ are related via \eqref{e:L_relationship}, 
we have the following result.

\begin{lemma}\label{L:L_ker}
Let $\phi$ be a non-trivial $T$-periodic solution of the profile equation \eqref{e:profile}.  So long as $T_a\neq 0$, we have 
\[
\ker_{\rm L^2_{\rm per}(0,T)}\left(\mathcal{L}[\phi]\right)={\rm span}\left\{\phi'\right\}\quad{\rm and}
	\quad\ker_{\rm L^2_{\rm per}(0,T)}\left(\mathcal{L}^\dag[\phi]\right)={\rm span}\left\{\phi^{-3}\phi'\right\}.
\]
Under the same assumption, we also have
\[
\ker_{\rm L^2_{\rm per}(0,T)}\left(A^\dag[\phi]\right)={\rm span}\left\{1,G^\dag[\phi]\phi^{-2}\right\}.
\]
\end{lemma}

\begin{proof}
Recalling $\phi$ may be chosen to be even, we have that $\phi'$ and $\phi_a$ are odd and even functions of $z$, respectively, so it follows from \eqref{e:L_ident1} that $\phi'$ and $\phi_a$ provide
two linearly independent solutions of the second order differential equation $\mathcal{L}[\phi]f=0$.  To identify the kernel of $\mathcal{L}[\phi]$, we must impose
$T$-periodic boundary conditions.  Since $\phi'$ is clearly $T$-periodic, it follows that the $T$-periodic kernel of $\mathcal{L}[\phi]$
has dimension at least one.  However, the fact that $T$ depends on the parameter $a$ implies that the function $u_a$ is generally not periodic.  Indeed, 
differentiating the obvious relation
\[
\left(\begin{array}{c}\phi(0;a,E,c)\\ \phi'(0;a,E,c)\end{array}\right)=\left(\begin{array}{c}\phi(T;a,E,c)\\ \phi'(T;a,E,c)\end{array}\right)
\]
with respect to $a$ gives 
\[
\left(\begin{array}{c}\phi_a(0)\\ \phi_{ax}(0) \end{array}\right)-\left(\begin{array}{c}\phi_a(T)\\ \phi_{ax}(T) \end{array}\right)
=T_a\left(\begin{array}{c}\phi'(T)\\ \phi''(T) \end{array}\right)
\]
which, using that $\mathcal{L}[\phi]$ is a second order differential equation and that $\phi$ is non-trivial\footnote{Note since $\phi$ satisfies the second order ODE
\eqref{e:profile}, vanishing of the vector $(\phi'(T),\phi''(T))^T$ would imply $\phi$ is the the trivial solution by uniqueness.  Alternatively, note that while $\phi'(T)=0$
by normalization, a direct calculation from \eqref{e:quad} shows that $\phi''(T) = -V_\phi(\phi_{\rm min};a,c)$, which again is non-zero since
$\phi$ is not an equilibrium solution of \eqref{e:profile}.},
 implies that $\phi_a$ is $T$-periodic if and only if $T_a$ is zero.  
This yields the characterization of the $T$-periodic kernel of $\mathcal{L}[\phi]$, and the kernel of $\mathcal{L}^\dag[\phi]$
follows immediately from \eqref{e:L_relationship}.

Finally, to characterize the $T$-periodic kernel of $A^\dag[\phi]$, observe that
\[
A^\dag[\phi]=-\mathcal{L}^\dag[\phi]\partial_x(G^{-1})^\dag[\phi].
\]
Since the adjoint of $G[\phi]$ is invertible on the space of $T$-periodic functions, and recalling $G^\dag[\phi]1=1$, the claim now follows from the characterization of the kernel of $\mathcal{L}^\dag[\phi]$.
\end{proof}

Next, we use Lemma \ref{L:L_ker} along with the Fredholm alternative to identify, under appropriate genericity conditions, the $T$-periodic generalized kernels of $A[\phi]$
and $A^\dag[\phi]$.  To this end, observe that \eqref{e:L_ident1} implies that
\[
A[\phi]\{\phi',\phi_a,\phi_E\}=0~~{\rm and}~~A[\phi]\phi_c=-\phi',
\]
which, among other things, yields three linearly independent functions satisfying the third order ODE $\partial_x\mathcal{L}[\phi]f=0$.  In Lemma \ref{L:L_ker} above, 
we showed that $\phi_a$ is not $T$-periodic provided that $T_a$ is non-zero.  Using similar arguments, it is readily seen that the functions 
$\phi_E$ and $\phi_c$ are not $T$-periodic provided that $T_E$ and $T_c$ are non-zero, respectively.  Indeed, we find that
\[
\left(\begin{array}{c}\phi_E(0)\\ \phi_{Ex}(0) \\ \phi_{Exx}(0)\end{array}\right)-\left(\begin{array}{c}\phi_E(T)\\ \phi_{Ex}(T) \\ \phi_{Exx}(T)\end{array}\right)
=T_E\left(\begin{array}{c}\phi'(T)\\ \phi''(T) \\ \phi'''(T)\end{array}\right),
\]
with an analogous equation holding for $\phi_c$.  Recalling by the above discussion that $\phi''(T)$ is non-zero, the desired result follows.  For notational simplicity, 
we introduce the following Poisson bracket style notation for two-by-two Jacobian determinants
\[
\{F,G\}_{x,y}:=\det\left(\begin{array}{cc} F_x & F_y\\ G_x & G_y\end{array}\right)
\]
and an analogous notation for three-by three Jacobian determinants:
\[
\{F,G,H\}_{x,y,z}:=\det\left(\begin{array}{ccc} F_x & F_y & F_z\\ G_x & G_y & G_z \\ H_x & H_y & H_z\end{array}\right).
\]
Using the above identities, it follows that, while $\phi_a$ and $\phi_E$ are not individually $T$-periodic, the linear combination 
\[
T_a\phi_E-T_E\phi_a=\{T,\phi\}_{a,E}
\]
lies in the $T$-periodic kernel of $A[\phi]$.  Similarly, we see that the functions $\{T,\phi\}_{a,c}$ and $\{T,\phi\}_{E,c}$ are both $T$-periodic and
satisfy
\begin{equation}\label{e:Jordan1}
A[\phi]\{T,\phi\}_{a,c}=-T_a\phi'\quad{\rm and}\quad A[\phi]\{T,\phi\}_{E,c}=-T_E\phi'.
\end{equation}
We now state the main result for this section.

\begin{theorem}\label{T:gker}
Let $\phi=\phi(\cdot;a,E,c)$ be a $T$-periodic solution of the profile equation \eqref{e:profile}, and assume the Jacobians $T_a$,
$\{T,M\}_{a,E}$ and $\{T,M,Q\}_{a,E,c}$ are non-zero.  Then $\lambda=0$ is an eigenvalue of the Bloch operator
$A_0[\phi]$ with algebraic multiplicity three and geometric multiplicity two.  In particular, defining 
\[
\begin{matrix}
\Phi_1 := \{T,M\}_{a,E}\phi ' &\hskip 5pt& \Phi_2 := \{T,\phi\}_{a,E} &\hskip 5pt& \Phi_3 :=  \{T,M,\phi\}_{a,E,c}\\
\Psi_1 :=  \eta &\hskip 5pt& \Psi_2 := 1 &\hskip 5pt& \Psi_3 := -\{T,M\}_{a,E}\Gadj\phi^{-2} - \{T,Q\}_{a,E},
\end{matrix}
\]  
where $\eta\in L^2_{\rm per}(0,T)$ is the unique odd function satisfying $A_0^\dag[\phi]\eta=-\Psi_3$,
we have that $\{\Phi_\ell\}_{\ell=1}^3$ and $\{\Psi_j\}_{j=1}^3$ provide a biorthogonal bases for the generalized kernels of
of $A_0[\phi]$ and $A_0^\dag[\phi]$, respectively.    In particular, we have $\langle \Psi_j, \Phi_\ell\rangle =0$ if and only if $j\neq \ell$.
Furthermore, the functions $\Phi_\ell$ and $\Psi_j$ satisfy the equations
\[
A_0[\phi]\Phi_1 = 0 = A_0[\phi]\Phi_2,~~A_0[\phi]\Phi_3 = -\Phi_1,
\]
and
\[
A_0^\dag[\phi] \Psi_2 = 0 = A_0^\dag[\phi] \Psi_3, ~~A_0^\dag[\phi] \Psi_1 = -\Psi_3.
\]
\end{theorem}

\begin{proof}
Since $T_a\neq 0$ by assumption, the characterization of the kernel of $A_0^\dag[\phi]$ follows from Lemma \ref{L:L_ker}.  Further, note a function $f$ belongs to the $T$-periodic
kernel of $A_0[\phi]$ if and only if $f$ is $T$-periodic and either
$\mathcal{L}[\phi]f=0$ or $\mathcal{L}[\phi]f$ is a non-zero constant.  From \eqref{e:L_ident1}, it follows immediately
that 
\[
\ker_{L^2_{\rm per}(0,T)}\left(A_0[\phi]\right)={\rm span}\left\{\Phi_1,\Phi_2\right\}.
\]
Furthermore, $\phi'$ is in the range of $A_0[\phi]$  by \eqref{e:Jordan1}.  Hence, by the Fredholm alternative (or by parity), we have that $\LA \Psi_2, \Phi_1\RA = 0 = \LA \Psi_3, \Phi_1\RA$.  For the periodic element that lies in the Jordan chain above $\phi'$, we take a specific linear combination of $\{T,\phi\}_{a,c}$ and $\{T,\phi\}_{E,c}$, namely $\Phi_3$.  
Furthermore, by the Fredholm alternative the Jordan chain above $\phi'$ terminates at height one if and only if
\[
\Phi_3\notin\ker\left(A_0^\dag[\phi]\right)^\perp.
\]
Since $\Phi_3$ has zero mean by construction, we clearly have $\left<\Psi_2,\Phi_3\right>=0$ and hence the above condition is equivalent to showing 
\begin{align*}
\left<\Psi_3,\Phi_3\right>=-\{T,M\}_{a,E}\left<G^\dag[\phi]\phi^{-2},\{T,M,\phi\}_{a,E,c}\right>
\end{align*}
is non-zero.  To write the above in a more usable form, observe from the definition of $Q$ in \eqref{e:MQ} that
\[
Q_a=\int_0^T\left(\frac{2\phi\phi'\phi'_a-\phi\phi_a-2(\phi')^2\phi_a}{\phi^3}\right)dx+\frac{[\phi(T)+(\phi'(T))^2]T_a}{(\phi(T))^2}.
\]
Note by integration by parts that
\[
\int_0^T\left(\frac{\phi'}{\phi^2}\right)\phi'_a dx=-\int_0^T\left(\frac{\phi\phi''-2(\phi')^2}{\phi^3}\right)\phi_a~dx
\]
and hence
\begin{equation}\label{e:Q_variation}
\begin{aligned}
Q_a&=-\int_0^T\left(\frac{\phi+2\phi\phi''-2(\phi')^2}{\phi^3}\right)\phi_a~dx+\frac{[\phi(T)+(\phi'(T))^2]T_a}{(\phi(T))^2}\\
&=-\left<G^\dag[\phi]\phi^{-2},\phi_a\right>+\frac{[\phi(T)+(\phi'(T))^2]T_a}{(\phi(T))^2}.
\end{aligned}
\end{equation}
Similar expressions hold for $Q_E$ and $Q_c$ and hence we find
\[
\left<\Psi_3,\Phi_3\right>=\{T,M\}_{a,E}\{T,M,Q\}_{a,E,c},
\]
which is non-zero by hypotheses.  The proves our characterization of the generalized $T$-periodic kernel of the operator $A_0[\phi]$.

Finally, we consider the generalized kernel of the adjoint operator $A_0^\dag[\phi]$.  Following the method for calculating $\LA\Psi_3,\Phi_3\RA$ above,
we immediately find that $\LA \Psi_2,\Phi_2\RA = \{T,M\}_{a,E}$, which is assumed to be non-zero.  Hence, by the Fredholm alternative, $\Psi_2$ is not in the range of 
$A_0^\dag[\phi]$.  By similar arguments, we find
\[
\left<\Psi_3,\Phi_2\right>=-\{T,M\}_{a,E}\left(\left<G^\dag[\phi]\phi^{-2},\{T,\phi\}_{a,E}\right>+\{T,Q\}_{a,E}\right)=0,
\]
so that, again by the Fredholm alternative, $\Psi_3$ belongs in the range of $A_0^\dag[\phi]$.  Since $\Phi_3$ is even and $A_0[\phi]$ switches parity,
the fact that the kernel of $A_0[\phi]$ consists entirely of even functions implies there exists a unique $T$-periodic odd function $\eta$ that satisfies
\[
A_0^\dag[\phi]\eta=-\Psi_3.
\]
Furthermore, we note that $\Psi_1 = \eta$ is not in the range of $A_0^\dag[\phi]$ since
\begin{align*}
\left<\Psi_1,\Phi_1\right>&=-\left<\Psi_1,A_0[\phi]\Phi_3\right>=\left<\Psi_3,\Phi_3\right>=\{T,M\}_{a,E}\{T,M,Q\}_{a,E,c},
\end{align*}
which is non-zero by assumption.  This completes the characterization of the generalized kernel of $A_0^\dag[\phi]$.  To finish the proof, we note that $\LA \Psi_1, \Phi_2\RA = 0 = \LA \Psi_1, \Phi_3\RA$ by parity.
\end{proof}

We now make some important comments regarding the assumptions in Theorem \ref{T:gker}.  The above result was obtained through the observation that infinitesimal variations
along the manifold of periodic traveling wave solutions yield tangent vectors that lie in the generalized kernels.  Through our existence theory in Section \ref{S:exist}, this manifold
was parameterized by the wave speed $c$ and the integration constants $a$ and $E$.  While this parameterization is natural from the mathematical perspective, 
following directly from the Hamiltonian formulation \eqref{e:quad} of the profile equation \eqref{e:profile}, 
it is different than the parameteriztaion that naturally arises in Whitham theory.  Indeed, recall from  Section \ref{S:whithamderivation} that the Whitham modulation
system \eqref{e:whitham} describes the slow evolution of the wave number $k$ and conserved quantities $M$ and $Q$, thus yielding a parameterization
of the manifold of periodic traveling wave solutions of the conduit equation \eqref{e:conduit1} in terms of these physical quantities.
Consequently, a-priori these two approaches work with different parameterizations of the same manifold.

In order to make comparisons between these two approaches, it is natural to assume that we can smoothly change between these parameterizations.
Specifically, we require that the manifold of periodic traveling wave solutions of \eqref{e:conduit1} constructed
in Section \ref{S:exist} can be locally reparameterized in a $C^1$ manner by the wave number $k$ and the conserved quantities $M$ and $Q$, i.e. that the mapping
\[ 
\mathcal{B}\ni (a,E,c)\mapsto (k(a,E,c),M(a,E,c),Q(a,E,c))\in\RM^3
\]
is locally $C^1$-invertible  at each point.  By the Implicit Function Theorem, this is guaranteed by requiring that the Jacobian determinant
\[
\frac{\partial(k,M,Q)}{\partial(a,E,c)}=\{k,M,Q\}_{a,E,c}
\]
of the above map is non-singular at each point in $\mathcal{B}$.  Recalling that $k(a,E,c)=\frac{1}{T(a,E,c)}$ we see that
\[
\{k,M,Q\}_{a,E,c}=-\frac{1}{T^2}~\{T,M,Q\}_{a,E,c},
\]
it follows that such a $C^1$-reparameterization is possible provided $\{T,M,Q\}_{a,E,c}\neq 0$, which is one of the primary assumptions in Theorem \ref{T:gker}.  Likewise,
the requirement that $\{T,M\}_{a,E}\neq 0$ is equivalent to saying that waves with fixed wave speed can be locally reparameterized in a $C^1$ manner by the
wave number $k$ and mass $M$.  

With the above observations in mind, we now seek a restatement of Theorem \ref{T:gker} that is generated with respect to infinitesimal variations along the $k$, $M$, and $Q$ coordinates.
To see the dependence on the wave number $k$ explicitly, we begin by rescaling the spatial variable as $y=kx$ and note that $T$-periodic traveling wave solutions
of \eqref{e:conduit1} correspond to $1$-periodic traveling wave solutions (with the same period) of the rescaled evolution equation
\begin{equation}
u_t + k(u^2)_y - k^2(u^2(u^{-1}u_t)_y)_y =0\label{e:conduit4}
\end{equation}
where $k=1/T$.  After rescaling the traveling coordinate $\theta = kz = k (x-ct) = y-\o t$, where $\o = kc$ is temporal frequency, it is readily seen that traveling wave solutions of \eqref{e:conduit4} correspond to solutions of the form $u(y,t) = \phi(y-\o t)$.  Hence, $\phi(\theta)$ is a stationary $1$-periodic  solution of the evolutionary equation
\begin{equation}\label{e:travel_conduit2}
u_t - kc u_\theta + k(u^2)_\theta - k^2(u^2(u^{-1}(u_t - kcu_\theta)_\theta)_\theta =0
\end{equation}
The rescaled profile equation now reads
\begin{equation}\label{e:profile_rescaled}
2cE = c\phi - \phi^2 + k^2 c (\phi ')^2 - k^2 c \phi \phi '',
\end{equation}
where now $'$ denotes differentiation with respect to $\theta$.  Through this rescaling,
the $T$-periodic solutions $\phi$ of \eqref{e:profile} now correspond to $1$-periodic solutions of \eqref{e:profile_rescaled}
for some $k$.  Similarly, the conserved quantities evaluated on the manifold of $1$-periodic solutions of \eqref{e:profile_rescaled} take the form
\begin{equation}\label{e:MQ_rescaled}
\left\{\begin{aligned}
M&:= \int_0^1 \phi(\th;k,M,Q)\, d\th\\
Q &:= \int_0^1 \frac{\phi(\th;k,M,Q)+ k^2(\phi '(\th;k,M,Q))^2}{(\phi(\th;k,M,Q))^2}\, d\th
\end{aligned}\right.
\end{equation}
while the linearization of \eqref{e:travel_conduit2} now reads
\[
G[\phi] v_t = k\d_\th \mathcal{L}[\phi] v
\]
where, with a slight abuse of notation, $\G$ and $\L$ are defined as in \eqref{ops}, albeit with $k\d_\th$ replacing $\d_z$, i.e.
\begin{equation}\label{ops_rescaled}
\left\{\begin{aligned}
\G f &= f - k^2 (\phi^2(\phi^{-1} f)')' = f + k^2 \phi '' f - k^2 \phi f ''\\
\mathcal{L}[\phi]f&=cf - 2\phi f - k^2 c \phi '' f + 2k^2 c \phi ' f ' - k^2 c \phi f ''
\end{aligned}\right.
\end{equation}
Clearly, the rescaled operator $G[\phi]$ is invertible as before, leading us to the consideration of the spectral problem 
\begin{equation}\label{lin_rescaled}
A[\phi]v=\lambda v
\end{equation}
posed on $L^2(\RM)$, where now, again with a slight abuse of notation, $\A = \Ginv k\d_\th \L$ is a linear operator with $1$-periodic
coefficients and $G[\phi]$ and $\L$ are differential operators as in \eqref{ops_rescaled}.

By Theorem \ref{T:gker}, we know that $\lambda=0$ is a $1$-periodic eigenvalue of the rescaled operator $A[\phi]$ in \eqref{lin_rescaled} with algebraic multiplicity three and geometric
multiplicity two.  In fact, differentiating \eqref{e:profile_rescaled} with respect to $x$, $M$, and $Q$ yields
\[
A[\phi]\phi'=0,\quad A[\phi]\phi_M=-kc_M\phi',\quad A[\phi]\phi_Q=-kc_Q\phi'
\]
Observe that since the profiles $\phi(\cdot;k,M,Q)$ are always $1$-periodic by construction, it follows the functions $\phi_M$ and $\phi_Q$ are also $1$-periodic.
In particular, the functions $\phi'$ and $c_M\phi_Q-c_Q\phi_M$ are linearly independent and span the $1$-periodic  kernel of $A[\phi]$.
We also mention that, as in Lemma \ref{L:L_ker}, the functions $1$ and $G^\dag[\phi]\phi^{-2}$ span the $1$-periodic kernel of the rescaled operator $A^\dag[\phi]$.
Furthermore, using similar arguments as in \eqref{e:Q_variation} we see that
\[
\LA1,\phi_M\RA = 1 = -\LA G^\dag[\phi]\phi^{-2},\phi_Q\RA
\]
and
\[
\LA 1,\phi_Q\RA = 0 = \LA G^\dag[\phi]\phi^{-2},\phi_M\RA.
\]
We also have
\[
\LA 1, \phi ' \RA = 0 = \LA G^\dag[\phi]\phi^{-2},\phi'\RA
\]
by parity.  Consequently, there is a linear combination of the functions $1$ and $G^\dag[\phi]\phi^{-2}$ that is orthogonal to the $1$-periodic kernel of $A[\phi]$.  It follows
there is a $1$-periodic function in the generalized kernel of $A^\dag[\phi]$.  Letting $A_\xi[\phi]$ denote the Bloch operators associated with
$A[\phi]$, defined now for $\xi\in[-\pi,\pi)$, we have the following result.

\begin{corollary}\label{C:gker_rescaled}
Let $\phi=\phi(a,E,c)$ be a $T$-periodic solution of the profile equation \eqref{e:profile}, and assume that the Jacobian determinants $T_a$, $\{T,M\}_{a,E}$
and $\{T,M,Q\}_{a,E,c}$ are all non-zero.  Then $\lambda=0$ is an eigenvalue of $A_0[\phi]$ with algebraic multiplicity three and geometric multiplicity two.  
In particular, defining
\[
\begin{matrix}
\Phi_1^0 := \phi ' &\hskip 5pt& \Phi_2^0 := \phi_M &\hskip 5pt& \Phi_3^0 := \phi_Q\\
\Psi_1^0 := \beta &\hskip 5pt& \Psi_2^0 := 1 &\hskip 5pt& \Psi_3^0 := -\Gadj\phi^{-2},
\end{matrix}
\]  
where $\beta\in L^2_{\rm per}(0,1)$ is the unique odd function satisfying $A_0^\dag[\phi]\beta\in\ker\left(A_0^\dag[\phi]\right)$ and $\left<\beta,\Phi_1^0\right>=1$,
we have that $\{\Phi_\ell^0\}_{\ell=1}^3$ and $\{\Psi_j^0\}_{j=1}^3$ provide a basis of solutions for the generalized kernels of $A_0[\phi]$ and $A_0^\dag[\phi]$, respectively.
In particular, we have $\LA \Psi_j^0,  \Phi_\ell^0 \RA = \delta_{j\ell}$ and the $\Phi_\ell^0$ and $\Psi_j^0$ satisfy the equations
\[
A_0[\phi]\Phi_1^0 = 0,~~  A_0[\phi]\Phi_2^0 = -kc_M \Phi_1^0,~~ A_0[\phi]\Phi_3^0 = -kc_Q \Phi_1^0
\]
and
\[
A_0^\dag[\phi]\Psi_2^0 =0=A_0^\dag[\phi]\Psi_3^0, ~~A_0^\dag[\phi]\Psi_1^0 \in{\rm span}\{\Psi_2^0,\Psi_3^0\}\setminus\{0\}.
\]
\end{corollary}

Before continuing, we note for future use that the function $\phi_k$ is also $1$-periodic and satisfies the equation
\begin{equation}\label{e:phi_k_identity1}
A[\phi]\phi_k=-kc_k\phi'-2k^2cG^{-1}[\phi]\left((\phi')^2-\phi\phi''\right)'.
\end{equation}
Furthermore, differentiating \eqref{e:MQ_rescaled} with respect to $k$, while holding $M$ and $Q$ constant, yields the important relations
\begin{equation}\label{e:phi_k_identity2}
\left<1,\phi_k\right>=0\quad{\rm and}\quad
\LA \Gadj\phi^{-2}, \phi_k\RA = 2k \LA \phi^{-2}, (\phi ')^2\RA.
\end{equation}

\subsection{Modulational Stability Calculation}

Now that we have constructed a basis for the generalized kernels of $A_0[\phi]$ and its adjoint in a coordinate system compatible with the Whitham system \eqref{e:whitham},
we proceed
to study how this triple eigenvalue bifurcates from the $(\lambda,\xi)=(0,0)$ state.
To this end, recall from \eqref{bloch_spec} that this is equivalent to seeking the $1$-periodic eigenvalues in a neighborhood of the origin of the associated Bloch operators
\[
A_\xi[\phi]=G_\xi^{-1}[\phi]k(\partial_\theta+i\xi)\mathcal{L}_\xi[\phi]
\]
for $|\xi|\ll 1$, where here $G_\xi[\phi] := e^{-i\xi\th}\G e^{i\xi\th}$ and $\mathcal{L}_\xi[\phi] := e^{-i\xi\th}\mathcal{L}[\phi] e^{i\xi\th}$
are the Bloch operators associated with the operators $G[\phi]$ and $\mathcal{L}[\phi]$ defined in \eqref{ops_rescaled}.
Following \cite{BrHJ16,BrJZ,BNR14}, we begin expanding the Bloch operators for $|\xi|\ll 1$.  
As these expressions are analytic
in $\xi$, it is straightforward to verify that
\begin{align*}
\mathcal{L}_\xi[\phi]=L_0+(ik\xi) L_1+(ik\xi)^2L_2~~{\rm and}~~G_\xi[\phi]=\widetilde{G}_0+(ik\xi)\widetilde{G}_1+(ik\xi)^2\widetilde{G}_2,
\end{align*}
where here
\[
L_0:=\mathcal{L}[\phi],~~L_1:=2kc\left(\phi'-\phi\partial_\theta\right),~~~~L_2:=-c\phi 
\]
and
\[
\widetilde G_0  := \G , ~~~~ \widetilde G_1 := -2\phi k\d_\th , ~~~~ \widetilde G_2  := -\phi .
\]
are operators acting on $L^2_{\rm per}(0,1)$.
Using that $\widetilde G_0[\phi]$ is invertible on $L^2_{\rm per}(0,1)$ and rewriting the expansion for $G_\xi[\phi]$ as
\[
G_\xi[\phi]=\left(I+(ik\xi)\widetilde{G}_1\widetilde{G}_0^{-1}+(ik\xi)^2\widetilde{G}_2\widetilde{G}_0^{-1}\right) \widetilde{G}_0,
\]
it follows $G_\xi^{-1}[\phi]$ can be expanded for $|\xi|\ll 1$ as the Neumann series
\begin{align*}
G_\xi^{-1}[\phi]&=\widetilde G_0^{-1}\sum_{\ell=0}^\infty (-1)^\ell[((ik\xi)\widetilde G_1 + (ik\xi)^2\widetilde G_2)\widetilde G_0^{-1}]^\ell\\
&= \mathcal{G}_0 + (ik\xi)\mathcal{G}_1 + (ik\xi)^2 \mathcal{G}_2 + \mathcal O(|\xi|^3),
\end{align*}
where here
\[
\mathcal{G}_0= G^{-1}[\phi],~~~\mathcal{G}_1 = 2\Ginv(\phi k\d_\th(\Ginv\, \cdot\,))
\]
and
\[
\mathcal{G}_2 = \Ginv(\phi\Ginv\, \cdot\,) + 4\Ginv(\phi k\d_\th(\Ginv(\phi k\d_\th(\Ginv \,\cdot\,))))
\]
are again acting on $L^2_{\rm per}(0,1)$.
The Bloch operators $A_\xi[\phi]$ can thus be expanded for $|\xi|\ll 1$ as
\begin{align*}
A_\xi[\phi] &= (\mathcal{G}_0 + (ik\xi)\mathcal{G}_1 + (ik\xi)^2\mathcal{G}_2 + \mathcal O(|\xi|^3))(k\d_\th + ik\xi)(L_0 + (ik\xi) L_1+(ik\xi)^2 L_2)\\
&= A_0 + (ik\xi) A^{(1)} + (ik\xi)^2 A^{(2)} + \mathcal O(|\xi|^3)
\end{align*}
where, after some manipulation, 
\begin{equation}\label{e:A_expand}
\left\{\begin{aligned}
A_0 &=A_0[\phi],\\
 A^{(1)} &= \mathcal{G}_1 k\d_\th L_0 + \mathcal{G}_0 k\d_\th L_1 + \mathcal{G}_0 L_0\\
&= \mathcal{G}_0 \left( 2 \phi k \d_\th A_0 + k \d_\th L_1 + L_0\right),\\
A^{(2)} &= \mathcal{G}_2 k\d_\th L_0 + \mathcal{G}_1 k\d_\th L_1 + \mathcal{G}_0 k\d_\th L_2 + \mathcal{G}_1 L_0 + \mathcal{G}_0 L_1\\
&= \mathcal{G}_0 \left(\phi A_0 + 4 \phi k\d_\th \mathcal{G}_0 \phi k\d_\th A_0 + 2 \phi k \d_\th \mathcal{G}_0 k\d_\th L_1 + k\d_\th L_2 + 2 \phi k \d_\th \mathcal{G}_0 L_0 + L_1\right).
\end{aligned}
\right.
\end{equation}
Note to assist in our computations of the actions of $A^{(1)}$ and $A^{(2)}$ later, we have expanded these operators in \eqref{e:A_expand} as to identify any factors of $A_0$ 
present in them, as well as to pull out a global factor of $\mathcal{G}_0$.

Now, by Corollary \ref{C:gker_rescaled}, we know $\lambda=0$ is an isolated eigenvalue of $A_0[\phi]$ with algebraic multiplicity three. Since $A_\xi[\phi]$ is a relatively compact perturbation
of $A_0[\phi]$ for all $|\xi|\ll 1$ depending analytically on the Bloch parameter $\xi$, it follows that the operator $A_\xi[\phi]$ will have three eigenvalues $\{\lambda_j(\xi)\}_{j=1}^3$,
defined for $|\xi|\ll 1$, bifurcating from $\lambda=0$ for $0<|\xi|\ll 1$.  The modulational stability or instability of $\phi$ may then be determined by tracking
these three eigenvalues for $|\xi|\ll 1$.  To this end, observe we may use the bases identified in Corollary \ref{C:gker_rescaled} to build explicit rank $3$ spectral
projections onto the generalized kernels of $A_0[\phi]$ and $A_0^\dag[\phi]$.  By standard spectral perturbation theory (see, for example, 
Theorems 1.7 and 1.8 in \cite[Chapter VII.1.3]{K76}) 
these dual bases extend analytically into dual right and left bases $\{\Phi_\ell^\xi\}_{\ell=1}^3$ and $\{\Psi_j^\xi\}_{j=1}^3$ associated to the 
three eigenvalues $\{\lambda_j(\xi)\}_{j=1}^3$ near the origin that satisfy $\langle\Psi_j^\xi,\Phi_\ell^\xi\rangle=\delta_{j\ell}$ for all $|\xi|\ll 1$.
For $|\xi|\ll 1$, we may now construct $\xi$-dependent rank 3 eigenprojections 
\[
\Pi(\xi): L^2_{\rm per}(0,2\pi)\to\bigoplus_{j=1}^3{\rm ker}\left(A_\xi[\phi]-\lambda_j(\xi)I\right),~~
	\widetilde\Pi(\xi): L^2_{\rm per}(0,2\pi)\to\bigoplus_{j=1}^3{\rm ker}\left(A^\dag_\xi[\phi]-\overline{\lambda_j(\xi)}I\right),
\]
with ranges coinciding with the total left and right eigenspaces associated with the eigenvalues $\{\lambda_j(\xi)\}_{j=1}^3$.
Using this one-parameter family of eigenprojections, for each fixed $|\xi|\ll 1$ we can project the infinite dimensional
spectral problem for $A_\xi[\phi]$ onto the three-dimensional total eigenspace associated with the eigenvalues $\{\lambda_j(\xi)\}_{j=1}^3$.
In particular, the action of the operators $A_\xi[\phi]$ on this subspace can be represented by the $3\times 3$ matrix operator
\[
D_\xi:=\widetilde{\Pi}(\xi)A_\xi[\phi]\Pi(\xi)=\left(\LA\Psi_j^\xi,A_\xi[\phi]\Phi_\ell^\xi\RA\right)_{j,\ell=1}^3.
\]
It follows that for each $|\xi|\ll 1$ the eigenvalues $\lambda_j(\xi)$ correspond precisely to the values $\lambda$ where the matrix
$D_\xi-\lambda I$ is singular, where here $I$ denotes the identity matrix\footnote{Here, we are using that $\widetilde{\Pi}(\xi)\Pi(\xi)$ is the identity by construction.} on 
$\RM^3$.  In what follows, we aim to explicitly construct this matrix for $|\xi|\ll 1$.

First, we observe that Corollary \ref{C:gker_rescaled} implies that $D_\xi$ at $\xi=0$ is given by\footnote{The horizontal and vertical lines are included only to organize the $3\times 3$ matrices
 into sub-blocks.}
\[
D_0 = \left(\begin{array}{c|cc} 0 & -kc_M & -kc_Q \\ \hline  0 & 0 & 0 \\ 0 & 0 & 0 \end{array}\right).
\]
Expanding the right and left bases $\Phi_\ell^\xi$ and $\Psi_j^\xi$ as\footnote{Note the expansion in $ik\xi$, rather than simply $i\xi$, is natural due to our spatial rescaling.  Further, it is
consistent with how $A_\xi[\phi]$ naturally expands in the rescaled system.}
\[
\Phi_\ell^\xi = \Phi_\ell^0 + (ik\xi)\left(\tfrac{1}{ik}\d_\xi\left.\Phi_\ell^\xi\right|_{\xi=0}\right) + (ik\xi)^2\left(\tfrac{1}{(ik)^2}\d_\xi^2\left.\Phi_\ell^\xi\right|_{\xi=0}\right) + \mathcal O(|\xi|^3)
\]
and 
\[
\Psi_j^\xi = \Psi_j^0 + (ik\xi)\left(\tfrac{1}{ik}\d_\xi\left.\Psi_j^\xi\right|_{\xi=0}\right) + (ik\xi)^2\left(\tfrac{1}{(ik)^2}\d_\xi^2\left.\Psi_j^\xi\right|_{\xi=0}\right) + \mathcal O(|\xi|^3),
\]
yields an expansion of the matrix $D_\xi$ 
of the form 
\[
D_\xi=D_0+(ik\xi)D^{(1)}+(ik\xi)^2D^{(2)}+\mathcal{O}(|\xi|^3).  
\]
Note that, explicitly,
\[
D^{(1)} = \left(\LA \Psi_j^0, A_0 \tfrac{1}{ik}\d_\xi\left.\Phi_\ell^\xi\right|_{\xi=0} + A^{(1)} \Phi_\ell^0\RA + \LA \tfrac{1}{ik}\d_\xi\left.\Psi_j^\xi\right|_{\xi=0}, A_0 \Phi_\ell^0\RA\right)_{j,\ell=1}^3.
\]

To continue, we need information regarding $\d_\xi\Phi_1^\xi$ at $\xi=0$.  We claim that, up to harmless modifications of the basis functions used above, we can arrange
for this first order variation to be  exactly $\phi_k$, while simultaneously preserving biorthogonality of the bases up to $\mathcal O(|\xi|^2)$.
Indeed, observe that by differentiating the identity
$\Pi(\xi)A_\xi[\phi]\Phi_1^\xi=A_\xi[\phi]\Phi_1^\xi$ and evaluating at $\xi=0$ yields the relation
\[
\Pi(0)\left(A_0\left(\tfrac{1}{ik}\d_\xi\left.\Phi_1^\xi\right|_{\xi=0}\right) + A^{(1)}\Phi_1^0\right) + \tfrac{1}{ik}\d_\xi\left.\Pi(\xi)\right|_{\xi=0}A_0\Phi_1^0 = A_0\left(\tfrac{1}{ik}\d_\xi\left.\Phi_1^\xi\right|_{\xi=0}\right) + A^{(1)}\Phi_1^0.
\]
Recalling $A_0\Phi_1^0=0$, and noting that \eqref{e:phi_k_identity1} can be rewritten as $A^{(1)}\Phi_1^0 = - A_0 \phi_k - kc_k \Phi_1^0$, we find that
\[
\Pi(0)\left(A_0\left(\tfrac{1}{ik}\d_\xi\left.\Phi_1^\xi\right|_{\xi=0} - \phi_k\right)\right) = A_0\left(\tfrac{1}{ik}\d_\xi\left.\Phi_1^\xi\right|_{\xi=0} - \phi_k\right)
\]
which implies that $\frac{1}{ik}\d_\xi\left.\Phi_1^\xi\right|_{\xi=0}-\phi_k$ lies in the generalized kernel of $A_0[\phi]$.  Consequently, there exists constants $\{a_j\}_{j=1}^3$ 
such that
\[
\tfrac{1}{ik}\d_\xi\left.\Phi_1^\xi\right|_{\xi=0}=\phi_k+\sum_{j=1}^3a_j\Phi_j^0.
\]
Replacing $\Phi_1^\xi$ with
\[
\widetilde{\Phi}_1^\xi:=\Phi_1^\xi-(ik\xi)\sum_{j=1}^3a_j\Phi_j^0
\]
while simultaneously replacing $\Psi_j^\xi$ for $j=1,2,3$ with
\[
\widetilde{\Psi}_j^\xi:=\Psi_j^\xi+(ik\xi)a_j\Psi_1^\xi,
\]
we readily see that
\begin{equation}\label{e:ident_expand}
\widetilde{\Phi}_1^\xi=\Phi_1^0+(ik\xi)\phi_k+\mathcal{O}(|\xi|^2),~~I_\xi=\widetilde\Pi(\xi)\Pi(\xi)=I+\mathcal{O}(|\xi|^2),
\end{equation}
as claimed.  Note that $I_\xi$ is the $3\times 3$ matrix describing the action of the identity operator with respect to the modified bases.  For notational
simplicity, we will drop the tildes throughout the remainder and refer to these modified bases as simply $\{\Phi_\ell^\xi\}_{\ell=1}^3$ and  $\{\Psi_j^\xi\}_{j=1}^3$
With the above choices, the terms involving the variations in $\Psi_j^\xi$ can be directly computed.  For example, using Corollary \ref{C:gker_rescaled} we have
\[
 \LA \tfrac{1}{ik}\d_\xi\left.\Psi_j^\xi\right|_{\xi=0}, A_0 \Phi_2^0\RA=-kc_M \LA \tfrac{1}{ik}\d_\xi\left.\Psi_j^\xi\right|_{\xi=0}, \Phi_1^0\RA
 =kc_M \LA \Psi_j^\xi, \tfrac{1}{ik}\d_\xi\left.\Phi_1^0\right|_{\xi=0}\RA=kc_M\left<\Psi_j^0,\phi_k\right>,
\]
where the second equality follows since \eqref{e:ident_expand}(ii) implies
\[
0=\partial_\xi\left.\LA\Psi_j^\xi,\Phi_1^\xi\RA\right|_{\xi=0}=
\LA \d_\xi\left.\Psi_j^\xi\right|_{\xi=0}, \Phi_1^0\RA + \LA \Psi_j^\xi, \d_\xi\left.\Phi_1^0\right|_{\xi=0}\RA.
\]

Using the above modified bases, straight forward calculations yield
\[
D^{(1)} = \left(\begin{array}{c|cc} 
-kc_k & * & * \\ 
\hline  0 & \LA\Psi_2^0, A^{(1)} \Phi_2^0\RA  & \LA\Psi_2^0, A^{(1)} \Phi_3^0\RA \\ 
0 & \LA\Psi_3^0, A^{(1)} \Phi_2^0 + kc_M\phi_k\RA & \LA\Psi_3^0, A^{(1)} \Phi_3^0 + kc_Q\phi_k\RA\end{array}\right),
\]
where here ``$*$" is used to denote undetermined terms that, as we will see, are irrelevant to our calculation.  Similarly, one finds 
\[
D^{(2)} = \left(\begin{array}{c|cc} 
* & * & * \\ 
\hline \LA\Psi_2^0, A^{(2)} \Phi_1^0 + A^{(1)} \phi_k\RA  & * & * \\ 
\LA\Psi_3^0, A^{(2)} \Phi_1^0 + A^{(1)}\phi_k + kc_k\phi_k\RA & * & * \end{array}\right)
\]
Noticing the above implies the matrix $D_\xi$ satisfies
\[
D_\xi = \left(\begin{array}{c|ccc} \mathcal O(|\xi|) & & \mathcal O(1) & \\ \hline \\ \mathcal O(|\xi|^2) & & \mathcal O(|\xi|) & \\ \\ \end{array}\right),
\]
where the upper left block is $1\times 1$, 
it follows by standard arguments that the eigenvalues $\lambda_j(\xi)$ are at least $C^1$ in $\xi$, and can thus be written as $\lambda_j(\xi)=ik\xi\mu_j(\xi)$ for some continuous functions
$\mu_j$ defined for $|\xi|\ll 1$.  Further, introducing, for $0<|\xi|\ll 1$, the invertible 
matrix\footnote{Note that the conjugating by $S(\xi)$ effectively replaces the coefficient of the 
$\phi'$, corresponding to the local phase $\psi$, with the wave number $|\xi|\psi\sim\psi_x$.  This is reminiscent of the fact that the Whitham system \eqref{e:whitham} 
involves the wave number $\psi_x$, rather than the phase $\psi$.}
\[
S(\xi) := \left(\begin{array}{c|cc} ik\xi & 0 & 0 \\ \hline 0 & 1 & 0 \\ 0 & 0 & 1 \end{array}\right)
\]
and defining $\widehat D_\xi := \frac{1}{ik\xi}S(\xi) D_\xi S(\xi)^{-1}$ and $\widehat I_\xi := S(\xi)I_\xi S(\xi)^{-1}$, it follows $\widehat{D}_\xi$ and $\widehat{I}_\xi$ are analytic in $ik\xi$
and, at $\xi=0$, are given by $\widehat I_0 =I$ and
\begin{equation}\label{D0}
\widehat{D}_0 = 
\left(\begin{array}{ccc}
-kc_k & -kc_M & - kc_Q\\
\LA\Psi_2^0, A^{(2)} \Phi_1^0 + A^{(1)} \phi_k\RA &  \LA\Psi_2^0, A^{(1)} \Phi_2^0\RA &  \LA\Psi_2^0, A^{(1)} \Phi_3^0\RA \\
\LA\Psi_3^0, A^{(2)} \Phi_1^0 + A^{(1)}\phi_k + kc_k\phi_k\RA & \LA\Psi_3^0, A^{(1)} \Phi_2^0 + kc_M\phi_k\RA & \LA\Psi_3^0, A^{(1)} \Phi_3^0 + kc_Q\phi_k\RA
\end{array}\right)
\end{equation}
Furthermore, the $\mu_j(\xi)$ are the eigenvalues of the matrix $\widehat{D}_\xi$ since clearly
\begin{align*}
\det(D_\xi - \l(\xi)I_\xi) =(ik\xi)^3 \det(\widehat D_\xi - \mu(\xi)\widehat I_\xi)
\end{align*}

In summary, we have proven the following result.

\begin{theorem}\label{T:rigorous_MI}
Under the hypotheses of Corollary \ref{C:gker_rescaled},  in a sufficiently small neighborhood of the origin, the spectrum
of $A[\phi]$ on $L^2(\RM)$ consists of precisely three $C^1$ curves $\{\lambda_j(\xi)\}_{j=1}^3$ defined for $|\xi|\ll 1$ which can be expanded as
\[
\lambda_j(\xi)=ik\xi\mu_j(0)+o(|\xi|),~~j=1,2,3
\]
where the $\mu_j(0)$ are precisely the eigenvalues of the matrix $\widehat{D}_0$ above.  In particular, a necessary condition for $\phi$ to be a spectrally
stable solution of \eqref{e:conduit1} is that all the eigenvalues of $\widehat{D}_0$ are real.
\end{theorem}

Note the possible spectral instability predicted from Theorem \ref{T:rigorous_MI} is of modulational type, occurring near the origin in the spectral plane for side-band Bloch frequencies. 
In general, computing the eigenvalues of $\widehat{D}_0$ is a difficult task, requiring one to identify the above inner products in terms of known quantities:  see, for example, \cite{BrJK11,BrHJ16}.  
As we will see, however, this identification is not necessary in order to rigorously connect modulational instability of $\phi$ to the Whitham modulation system \eqref{e:whitham}.
Indeed, recalling the quasilinear form \eqref{e:whitham_quasilinear} of the Whitham modulation system \eqref{e:whitham}, in the next section we will prove that
\[
{\bf D}(\phi)=\widehat{D}_0-cI
\]
so that, in particular, a necessary condition for the spectral stability of $\phi$ is that the matrix ${\bf D}(\phi)$ is weakly hyperbolic, i.e. that all of its eigenvalues
are real, thus establishing Theorem \ref{T:main}.

\section{Proof of Theorem \ref{T:main}}\label{S:comparison}

In this section, we establish Theorem \ref{T:main}.  We will use a direct, row-by-row calculation to show that
\[
{\bf D}(\phi)=\widehat{D}_0-cI,
\]
where here ${\bf D}(\phi)$ is the linearized matrix associated to the Whitham modulation equations \eqref{e:whitham_quasilinear} and the eigenvalues of $\widehat{D}_0$, defined in \eqref{D0},
rigorously describe the structure of the $L^2(\RM)$-spectrum of the linearized operator $A[\phi]$ in a sufficiently small neighborhood of the origin: see Theorem \ref{T:rigorous_MI}.
To this end, first observe that the first rows of each of these matrices are clearly identical.   To compare the second rows, it is enough to establish the identities
\begin{equation}\label{e:2nd_row}
\left\{\begin{aligned}
&\LA 1, A^{(2)} \phi ' + A^{(1)} \phi_k\RA= \LA 1, 2k^2 c (\phi ')^2 - \phi^2 \RA_k,\\
&\LA 1, A^{(1)} \phi_M\RA-c= \LA 1, 2k^2 c (\phi ')^2 - \phi^2 \RA_M,\\
&\LA 1, A^{(1)} \phi_Q\RA= \LA 1, 2k^2 c (\phi ')^2 - \phi^2 \RA_Q.
\end{aligned}\right.
\end{equation}
For the first equality, observe from \eqref{e:A_expand} that, after some manipulations,
\begin{equation}\label{e:complicated1}
A^{(2)} \phi ' + A^{(1)} \phi_k = \Ginv\left( kc (\phi ')^2 - 3kc \phi \phi ''-2k^2 c_k \phi \phi '' + 2k^2 c\d_\th(\phi ' \phi_k - \phi \phi_k ') + \L \phi_k\right)
\end{equation}
which, recalling that $G^{-\dag}[\phi](1)=1$ and using integration by parts, yields
\begin{align*}
\LA 1, A^{(2)} \phi ' + A^{(1)} \phi_k \RA &= 
 \LA 1, 4kc (\phi ')^2 + 2k^2 c_k (\phi ')^2 + \L \phi_k \RA.
\end{align*}
Since integration by parts implies
\begin{align*}
\LA 1, \L \phi_k \RA &= \LA 1, c\phi_k - 2\phi \phi_k - k^2 c \phi '' \phi_k + 2k^2 c \phi ' \phi_k ' - k^2 c \phi \phi_k ''\RA\\
&= c\LA 1, \phi_k \RA + \LA 1, -(\phi^2)_k + 2k^2 c ((\phi ')^2)_k\RA
\end{align*}
it follows that
\[
\LA 1, A^{(2)}\phi ' + A^{(1)} \phi_k\RA = c\LA 1, \phi_k \RA + \LA 1, 2k^2 c (\phi ')^2 - \phi^2\RA_k,
\]
which, recalling $\LA 1, \phi_k\RA = 0$ by \eqref{e:phi_k_identity2}(i), establishes \eqref{e:2nd_row}(i).  The other two equalities in \eqref{e:2nd_row} follow in a similar way.  
Indeed, using analogous manipulations as above, as well as  integration by parts, we find 
\[
\LA 1, A^{(1)} \phi_M \RA  = 
 c\LA 1, \phi_M \RA+\LA 1, 2k^2 c (\phi ')^2 - \phi^2 \RA_M
\]
and
\[
\LA 1, A^{(1)} \phi_Q \RA
 =c\LA 1, \phi_Q \RA +\LA 1, 2k^2 c (\phi ')^2 - \phi^2 \RA_Q.
\]
Since $\LA 1, \phi_M\RA = 1$ and  $\LA 1, \phi_Q\RA = 0$ by the biorthogonality relations 
in Corollary \eqref{C:gker_rescaled}, the identities \eqref{e:2nd_row}(ii)-(iii)  follow.  This establishes that the second rows of the matrices $\bf{D}(\phi)$ and
$\widehat{D}_0-cI$ are identical.

To compare the third rows, we aim to establish the following three identities: 
\begin{equation}\label{e:3rd_row}
\left\{\begin{aligned}
&\LA G^\dag[\phi]\phi^{-2}, A^{(2)} \phi' + A^{(1)}\phi_k + kc_k\phi_k\RA=-\LA 1, 2\ln|\phi|\RA_k,\\
&\LA G^\dag[\phi]\phi^{-2}, A^{(1)} \phi_M + kc_M\phi_k\RA=-\LA 1, 2\ln|\phi|\RA_M\\
&\LA  \Gadj \phi^{-2}, A^{(1)} \phi_Q + kc_Q\phi_k\RA +c= -\LA 1, 2\ln|\phi|\RA_Q
\end{aligned}\right.
\end{equation}
Focusing on the first term, we note that \eqref{e:complicated1} and integration by parts implies
\begin{align*}
\LA  \Gadj \phi^{-2}, A^{(2)} \phi' + A^{(1)}\phi_k + kc_k\phi_k\RA &= \LA  \phi^{-2}, - 2k(c + k c_k)(\phi ')^2 + 2k^2 c(\phi '' \phi_k - \phi \phi_k '') + \L \phi_k\RA\\
&+ \LA \Gadj \phi^{-2}, kc_k\phi_k\RA
\end{align*}
so that, by \eqref{e:phi_k_identity2}(ii),
\[
\LA  \Gadj \phi^{-2}, A^{(2)} \phi' + A^{(1)}\phi_k + kc_k\phi_k\RA = - c \LA \Gadj\phi^{-2}, \phi_k\RA + \LA \phi^{-2}, 2k^2 c(\phi '' \phi_k - \phi \phi_k '') + \L \phi_k\RA.
\]
Now, using the fact that $\LA \phi^{-2}, \phi \phi_k ''\RA = \LA \phi^{-2}, \phi ' \phi_k '\RA$ we see
\begin{align*}
\LA \phi^{-2}, \L\phi_k\RA &= -\LA  \phi^{-2}, 2\phi \phi_k\RA+\LA\phi^{-2}, c\phi_k - k^2 c \phi '' \phi_k + 2k^2 c \phi ' \phi_k ' - k^2 c \phi \phi_k ''\RA\\
&=-2\left<1,\ln|\phi|\right>_k+\left<\phi^{-2},c\phi_k-k^2c\left(\phi''\phi_k-\phi\phi_k''\right)\right>
\end{align*}
and hence
\begin{align*}
\LA  \Gadj \phi^{-2}, A^{(2)} \phi' + A^{(1)}\phi_k + kc_k\phi_k\RA = &- c \LA \Gadj\phi^{-2}, \phi_k\RA - \left<1,2\ln|\phi|\right>_k\\
&+\left<\phi^{-2},c\phi_k+k^2c\left(\phi''\phi_k-\phi\phi_k''\right)\right>\\
= &- c \LA \Gadj\phi^{-2}, \phi_k\RA-\left<1,2\ln|\phi|\right>_k+c\left<\phi^{-2},G[\phi]\phi_k\right>.
\end{align*}
The identity \eqref{e:3rd_row}(i) has now been established.  The other two equalities follow similarly.  Indeed, using analogous manipulations 
as above, as well as integration by parts, we find that
\[
\LA G^\dag[\phi]\phi^{-2}, A^{(1)} \phi_M + kc_M\phi_k\right>=-\LA 1, 2\ln|\phi|\RA_M + c\left<G^\dag[\phi]\phi^{-2},\phi_M\right>
\]
and
\[
\LA G^\dag[\phi]\phi^{-2}, A^{(1)} \phi_Q + kc_Q\phi_k\right>=-\LA 1, 2\ln|\phi|\RA_Q + c\left<G^\dag[\phi]\phi^{-2},\phi_Q\right>.
\]
Recalling that $\left<G^\dag[\phi]\phi^{-2},\phi_M\right>=0$ and $\left<G^\dag[\phi]\phi^{-2},\phi_Q\right>=-1$ by the biorthogonality relations 
in Corollary \eqref{C:gker_rescaled}, the identities
\eqref{e:phi_k_identity2}(ii)-(iii) follow.  This complete the proof of Theorem \ref{T:main}

\section{Analysis for Small Amplitude Waves}\label{S:small}

In general, determining whether the Whitham system \eqref{e:whitham} is weakly hyperbolic at a given periodic traveling solution of \eqref{e:conduit1} is a difficult matter.  
See \cite{BrHJ16,BrJK11}, for example, for cases where this information can be computed in the presence of a Hamiltonian structure.
Nevertheless, one can use well-conditioned numerical methods to approximate the entries of the matrix ${\bf D}(\phi)$ in \eqref{e:whitham_quasilinear}, thereby 
producing a numerical stability diagram.  Such numerical analysis was recently carried out in \cite{MH16}.  There, the authors findings indicate, among other things, that
the Whitham modulation system \eqref{e:whitham} is hyperbolic about a given periodic traveling wave $\phi$ provided $\phi$ has sufficiently large period, while the system
is elliptic for sufficiently small periods.  Formally then, the authors findings suggest long waves are modulationally stable while short waves are modulationally unstable.  
For asymptotically small waves, however, it is possible to use asymptotic analysis to analyze the hyperbolicity of the Whitham modulation system \eqref{e:whitham}.
While this analysis is discussed in \cite{MH16}, for completeness we reproduce their result for small amplitude waves.

To this end, we note that $T=1/k$-periodic traveling wave solutions with speed $c$ of the conduit equation \eqref{e:conduit1} correspond to $1$-periodic stationary solutions
of the profile equation
\begin{equation}\label{e:profile_diff}
-\omega \phi ' + 2k \phi \phi ' + k^2 \omega \phi \phi ''' - k^2\omega \phi ' \phi '' = 0,
\end{equation}
where where $\omega=kc$ is the frequency and primes denote differentiation with respect to the traveling variable $\theta=kx-\omega t$.  
Using an elementary Lyapunov-Schmidt argument, one can show that solution pairs $(\phi,\omega)$ of \eqref{e:profile_diff} with asymptotically small oscillations about its
mean\footnote{Note that, here, the mean $M$ agrees with the mass, as defined in \eqref{e:mass}, since we are working with $1$-periodic functions.  More generally,
these quantities would differ by factor of $k$.}  $M$ admit a convergent asymptotic expansion in $H^3_{\rm per}(0,1)$ of the form
\begin{align*}
\phi(\theta;k,M,A) &= M + A \cos(2\pi\theta) + \sum_{j=2}^\infty A^j \phi_j(\theta;k,M)\\
\o(k,M,A) &= \o_0(k,M) + A^2 \o_2(k,M) + \mathcal{O}(A^4) 
\end{align*}
valid for $|A|\ll 1$, where the functions $\phi_j$ are $1$-periodic and satisfy 
\[
\int_0^1\phi_j(\theta)d\theta=0=\int_0^1\phi_j(\theta)\cos(2\pi\theta)d\theta
\]
for all $j\geq 2$.
Further, it is an easy calculation to see that 
\[
\phi_2(\theta) = \frac{1}{6(2\pi)^2k\o_0 M}\cos(4\pi\theta) = \frac{(2\pi)^2 k^2 M+1}{12(2\pi)^2 k^2 M^2}\cos(4\pi\theta)
\]
and
\[
\o_0 = \frac{2kM}{(2\pi)^2 k^2 M+1},\quad\o_2 = \frac{1 - 8(2\pi)^2 k^2 M}{12(2\pi)^2 k M^2((2\pi)^2 k^2 M+1)}.
\]
Using these asymptotic expansions, the Whitham modulation system \eqref{e:whitham} about these waves expands for $|A|\ll 1$ as 
\[
\left\{\begin{aligned}
&k_S+\partial_X\left(\omega_0+A^2\omega_2\right)=\mathcal{O}(A^3)\\
&M_S+\partial_X\left(M^2-\tfrac{1}{2}A^2\left(2(2\pi)^2k\omega_0-1\right)\right)=\mathcal{O}(A^3)\\
&\left(1+\mathcal{O}(A^2)\right) \left(A^2\right)_S + A^2\frac{\partial^2\omega_0}{\partial k^2}k_X + \left(A^2\right)_X\frac{\partial\omega_0}{\partial k}+4A^2\left(1-\frac{2(2\pi)^2k^2M}{(2(2\pi)^2k^2M+1)^3}\right)M_X=\mathcal{O}(A^3).
\end{aligned}
\right.
\]
Using the chain rule, this system can be rewritten in the quasilinear form
\[
\left(\begin{array}{c}k\\M\\A\end{array}\right)_S+B(k,M,A)\left(\begin{array}{c}k\\M\\A\end{array}\right)_X=0,
\]
where here $B$ expands in $A$ as 
\[
B(k,M,A)=B_0(k,M)+AB_1(k,M)+A^2\widetilde{B}(k,M,A),
\]
with $\widetilde{B}(k,M,A)$ a bounded, continuous matrix-valued function and
\[
B_0=\left(\begin{array}{ccc}
					\frac{\partial\omega_0}{\partial k} & \frac{\partial\omega_0}{\partial M} & 0\\
					0 & 2M & 0\\
					0 & 0 & \frac{\partial\omega_0}{\partial k}
					\end{array}\right),
~~
B_1=\left(\begin{array}{ccc}
					0 & 0 & 2\omega_2\\
					0 & 0 & 1-2(2\pi)^2k\omega_0\\
					\frac{1}{2}\frac{\partial^2\omega_0}{\partial k^2} & 2\left(1-\frac{2(2\pi)^2k^2M}{(2(2\pi)^2k^2M+1)^3}\right) & 0
					\end{array}\right).
\]

Based on the above asymptotics, the eigenvalues of the Whitham system \eqref{e:whitham} about the small amplitude $1$-periodic traveling wave $\phi(\cdot;k,M,A)$ may be asymptotically
expanded as
\[
\left\{\begin{aligned}
&\lambda_1(k,M,A)=2M+\mathcal{O}(A^2)\\
&\lambda_{\pm}(k,M,A)=\frac{\partial\omega_0}{\partial k}\pm A\sqrt{-n(k,M)\frac{\partial^2\omega_0}{\partial k^2}}+\mathcal{O}(A^2),
\end{aligned}\right.
\]
where we have explicitly
\[
n(k,M)=\frac{8[(2\pi)^2k^2M]^2+5[(2\pi)^2k^2M]+3}{12(2\pi)^2kM^2([(2\pi)^2k^2M]+1)([(2\pi)^2k^2M]+3)}.
\]
Since 
\[
\frac{\partial^2\omega_0}{\partial k^2}(k,M)=\frac{16\pi^2 kM^2\left((2\pi k)^2M-3\right)}{((2\pi k))^2M+1)^3},
\]
and since $n(k,M)$ is clearly a strictly positive function of $k$ and $M$, we immediately have the following result.

\begin{theorem}\label{T:small}
Let $\phi(\cdot;k,M,A)$ be a $1$-periodic traveling wave solution of \eqref{e:profile_diff} with asymptotically small amplitude.
Then $\phi$ is modulationally unstable if 
\[
(2\pi k)^2>\frac{3}{M}.
\]
Further, a necessary condition for $\phi$ to be modulationally stable is
\[
0<(2\pi k)^2<\frac{3}{M}.
\]
\end{theorem}

Theorem \ref{T:small} makes rigorous the formal calcluations in \cite{MH16} regarding small amplitude periodic traveling wave solutions\footnote{In \cite{MH16}, however,
note that $k$ was the wave number relative to $2\pi$-perturbations, while in our work $k$ denotes the wave number relative to $1$-periodic perturbations.  This accounts
for the extra factor of $2\pi$ present in our result compared to that in \cite{MH16}.} of \eqref{e:conduit1}.
Note that while the above analysis shows the Whitham system \eqref{e:whitham} is (strictly) hyperbolic when $0<(2\pi k)^2<3/M$, this is not sufficient
to conclude modulational stability of the underlying wave $\phi$ since hyperbolicity of \eqref{e:whitham} only guarantees the eigenvalues of $\mathcal{A}_\xi[\phi]$
lie on the imaginary axis to first order in $\xi$, i.e. it guarantees tangency of the spectral curves at $\lambda=0$ to the imaginary axis.  Of course, modulational
stability requires that the  spectral curves near the origin are confined to the left half plane, and hence cannot be concluded\footnote{An exception to this rule
occurs when the PDE admits a Hamiltonian strucutre.  See, for example, related work on the KdV equation \cite{BrJ10,BrHJ16}.} from only first order information.
Modulational stability was concluded in \cite{MH16} for the conduit equation, however, in the case $0<(2\pi k)^2<3/M$ through numerical time evolution.  
It would be interesting to rigorously verify this prediction.

\begin{remark}
Note that a slightly different approach to proving Theorem \ref{T:main} would have been to expand the matrix $D(\phi)$ in \eqref{e:whitham_quasilinear}, and using the chain
rule to express derivatives with respect to $Q$ in terms of derivatives with respect to $(k,M,A)$.  This calculation of course produces the same
result.  However, here we preferred to start working directly with the variables $(k,M,A)$ in the Whitham system, since this is the natural parameterization 
in the asymptoically small amplitude limit.
\end{remark}

\appendix

\bibliographystyle{plain}
\bibliography{Conduit}

\end{document}